\documentclass[12pt]{amsart}
\usepackage{amscd,amsmath,amsthm,amssymb}
\usepackage[left]{lineno}
\usepackage{color}
\usepackage{stmaryrd}
\usepackage[utf8]{inputenc}
\usepackage{cleveref}
\usepackage{epstopdf}
\usepackage{graphicx}
\usepackage{xcolor}

\usepackage{tikz}
\definecolor{verylight}{gray}{0.97}
\definecolor{light}{gray}{0.9}
\definecolor{medium}{gray}{0.85}
\definecolor{dark}{gray}{0.6}

\def\NZQ{\mathbb}               % the font for N,Z,Q,R,C

\def\KK{{\NZQ K}}

\def\KK{{\NZQ K}}

\newcommand\frakP{\mathfrak{P}}
\newcommand\frakQ{\mathfrak{Q}}

\def\G{{\mathcal G}}

\def\0b{{\mathbf 0}}

\def\reg{{\mathbf reg}}
\def\height{\operatorname{ht}}
\def\depth{\operatorname{depth}}
\def\opn#1#2{\def#1{\operatorname{#2}}} % to make operators
\opn\chara{char} \opn\length{\ell} \opn\pd{pd} \opn\rk{rk}
\opn\projdim{proj\,dim} \opn\injdim{inj\,dim} \opn\rank{rank}
\opn\depth{depth} \opn\grade{grade} \opn\height{height}
\opn\embdim{emb\,dim} \opn\codim{codim}

\opn\Tr{Tr} \opn\bigrank{big\,rank}
\opn\superheight{superheight}\opn\lcm{lcm}
\opn\trdeg{tr\,deg}%\emph{
	\opn\reg{reg} \opn\lreg{lreg} \opn\ini{in} \opn\lpd{lpd}
	\opn\size{size} \opn\sdepth{sdepth}
	\opn\link{link}\opn\fdepth{fdepth}\opn\lex{lex}
	\opn\tr{tr}
	\opn\type{type}
	\opn\gap{gap}
	\opn\arithdeg{arith-deg}
	\opn\HS{HS}
	\opn\GL{GL}
	\opn\div{div} \opn\Div{Div} \opn\cl{cl} \opn\Cl{Cl}
	\opn\Spec{Spec} \opn\Supp{Supp} \opn\supp{supp} \opn\Sing{Sing}
	\opn\Ass{Ass} \opn\Min{Min}\opn\Mon{Mon}
	\opn\Ann{Ann} \opn\Rad{Rad} \opn\Soc{Soc}\opn\Deg{Deg}
	\opn\Im{Im} \opn\Ker{Ker} \opn\Coker{Coker} \opn\Am{Am}
	\opn\Hom{Hom} \opn\Tor{Tor} \opn\Ext{Ext} \opn\End{End}
	\opn\Aut{Aut} \opn\id{id}
	
	\opn\nat{nat}
	\opn\pff{pf}%   \pf exists already
	\opn\Pf{Pf} \opn\GL{GL} \opn\SL{SL} \opn\mod{mod} \opn\ord{ord}
	\opn\Gin{Gin} \opn\Hilb{Hilb}\opn\sort{sort}
	\opn\PF{PF}\opn\Ap{Ap}
	\opn\mult{mult}
	\opn\bight{bight}
	\opn\aff{aff}
	\opn\relint{relint} \opn\st{st}
	\opn\lk{lk} \opn\cn{cn} \opn\core{core} \opn\vol{vol}  \opn\inp{inp} \opn\nilpot{nilpot}
	\opn\link{link} \opn\star{star}\opn\lex{lex}\opn\set{set}
	\opn\width{wd}
	\opn\Fr{F}
	\opn\QF{QF}
	\opn\G{G}
	\opn\type{type}\opn\res{res}
	\opn\conv{conv}
	\opn\Ind{Ind}
	\opn\gr{gr}
	
	\def\pot#1#2{#1[\kern-0.28ex[#2]\kern-0.28ex]}

	\opn\dirlim{\underrightarrow{\lim}}
	\opn\inivlim{\underleftarrow{\lim}}
	\def\Implies{\ifmmode\Longrightarrow \else
		\unskip${}\Longrightarrow{}$\ignorespaces\fi}
	\def\implies{\ifmmode\Rightarrow \else
		\unskip${}\Rightarrow{}$\ignorespaces\fi}
	\def\iff{\ifmmode\Longleftrightarrow \else
		\unskip${}\Longleftrightarrow{}$\ignorespaces\fi}

	\let\:=\colon
	\newtheorem{Theorem}{Theorem}[section]
	\newtheorem{Lemma}[Theorem]{Lemma}
	\newtheorem{Corollary}[Theorem]{Corollary}

	\newtheorem{Definition}[Theorem]{Definition}

	\let\epsilon\varepsilon
	\let\kappa=\varkappa
	\def\qed{\ifhmode\textqed\fi
		\ifmmode\ifinner\quad\qedsymbol\else\dispqed\fi\fi}
	\def\textqed{\unskip\nobreak\penalty50
		\hskip2em\hbox{}\nobreak\hfil\qedsymbol
		\parfillskip=0pt \finalhyphendemerits=0}
	\def\dispqed{\rlap{\qquad\qedsymbol}}
	
	\opn\dis{dis}
	\def\pnt{{\raise0.5mm\hbox{\large\bf.}}}
	
	\opn\Lex{Lex}
	
\begin{document}
		%\linenumbers
		\title {Depth of powers of  edge ideals of  edge-weighted integrally closed cycles}
		
		\author {Guangjun Zhu$^{\ast}$, Jiaxin Li, Yijun Cui and Yi Yang}

		\address{School of Mathematical Sciences, Soochow University, Suzhou, Jiangsu, 215006, P. R. China}

		\email{ zhuguangjun@suda.edu.cn(Corresponding author:Guangjun Zhu),\linebreak[4]
			237546805@qq.com(Yijun Cui), lijiaxinworking@163.com(Jiaxin Li), 3201088194@qq.com(Yi Yang).}
		%	\linebreak[4]
		
		\thanks{$^{\ast}$ Corresponding author}
		


		% \subjclass[2010]{Primary 13C99; Secondary 13E15, 13A15.}
		%		13H10   	Special types (Cohen-Macaulay, Gorenstein, Buchsbaum, etc.)
		%		13D02   	Syzygies, resolutions, complexes
		%		05E40   	Combinatorial aspects of commutative algebra
		%		16S36   	Ordinary and skew polynomial rings and semigroup rings
		
		%		14M25   	Toric varieties, Newton polyhedra [See also 52B20]
		%		13A02   	Graded rings
		%		13F20   	Polynomial rings and ideals; rings of integer-valued polynomials
		%		13A18   	Valuations and their generalizations
		%		06A11   	Algebraic aspects of posets
		%       05C38       Paths and cycles
		%       13A15       Ideals; multiplicative ideal theory
		%       13F20       Polynomial rings and ideals; rings of integer-valued
		
		%\keywords{   }
		
		\maketitle
		\begin{abstract}
			This paper gives some exact formulas for the depth  of powers of the edge ideal of an edge-weighted  integrally closed cycle.
		\end{abstract}
		\setcounter{tocdepth}{1}
		%\tableofcontents

		%	\section*{Introduction}
		\section{Introduction}
	Let $G$ be a finite simple graph with the vertex set $V(G)$ and  the edge set $E(G)$.
		We write $xy$ for  an  edge of $G$ with $x$ and $y$ as endpoints.
		Suppose ${\bf w}: E(G)\longrightarrow \mathbb{Z}_{>0}$ is an edge-weighted function (weighted function for short) on $G$. We write $G_{\bf w}$ for the pair $(G,{\bf w})$ and call it an {\em edge-weighted} graph  (or just  weighted graph) with the underlying graph  $G$. For a weighted graph $G_{\bf w}$ with the vertex set $V(G)=\{x_1,\ldots,x_n\}$,
its {\em edge-weighted ideal}  (edge ideal  for short), was introduced in \cite{PS}, is a monomial ideal of the polynomial ring $S=\KK[x_{1},\dots, x_{n}]$ in $n$ variables over a field $\KK$, given by
		\[
		I(G_{\bf w})=(x^{{\bf w}(xy)}y^{{\bf w}(xy)}\mid xy\in E(G_{\bf w})).
		\]
		If ${\bf w}$ is the constant function defined by ${\bf w}(e)=1$ for all $e\in E(G)$, then  $I(G_{\bf w})$ is the classical  edge ideal of the underlying graph $G$ of $G_{\bf w}$, which
		has been  studied extensively in the literature \cite{FM,HTV,HH1,HT,M,MV,T,Zh1,Zh2}.
		
	In \cite{B},  Brodmann demonstrated that for a proper ideal $I\subset S$,  $\depth(S/I^t)$ is a constant for $t \gg 0$, and this constant is bounded above by $n-\ell (I)$, where $\ell(I)$ is the analytic spread of $I$.
 Both  Eisenbud and Huneke  in \cite[Proposition 3.3]{EH} and Herzog and Hibi in \cite[Theorem 1.2]{HH1}  independently showed that the upper bound can be obtained when the associated graded ring  is Cohen–Macaulay.
Furthermore, Herzog and Hibi proved \cite{HH1}  that  $\depth(S/I^t)$ is a non-increasing function if all the powers of $I$ have a linear resolution. They also provided lower bounds for $\depth(S/I^t)$ when all the powers of $I$ have linear quotients, a condition that entails that all powers of  $I$ have linear resolutions. In particular, they showed that the edge ideal of a finite graph whose complementary graph is chordal has linear quotients. Morey in \cite{M} gives lower bounds on the depths of all powers when  $I$  is the edge ideal of a forest.  For the  general graph,   Fouli1 and Morey in \cite{FM} gave lower bounds for its edge ideal.
In contrast, little is known about the depth of the quotient ring $S/I^t$ when $I$ is the  edge ideal of a weighted graph.

Recently, there has been a surge of interest in characterizing weights for which the edge ideal of a weighted graph is Cohen-Macaulay. For example, Paulsen and Sather-Wagstaff in \cite{PS}  classified Cohen-Macaulay  weighted graphs $G_{\bf w}$ where the underlying graph $G$ is a  cycle, a tree, or a  complete graph.  Seyed Fakhari et al. in \cite{FSTY} continued this study and classified  a
Cohen-Macaulay  weighted graph $G_{\bf w}$ when $G$ is a very well-covered graph.  Hien in \cite{Hi} classified Cohen-Macaulay weighted graphs $G_{\bf w}$ when $G$ has girth at least $5$.	
  Recently, the first second authors of this paper and Trung  in  \cite{LTZ}  characterized the Cohen-Macaulayness of the second power of the edge ideal of $G_{\bf w}$  when the underlying graph $G$ is a very well-covered graph and all ordinary powers of  the edge ideal of $G_{\bf w}$ when $G$ is a tree with a perfect matching consisting of pendant edges.  In \cite{W},  Wei classified all Cohen-Macaulay weighted chordal graphs from a purely graph-theoretic point of view.  Diem et al.  in \cite{DMV} gave a  complete  characterization of sequentially Cohen-Macaulay edge-weighted graphs.
  The first three authors of this paper and Duan in \cite{ZDCL} gave some exact formulas for the depth of the powers of the edge ideal of a weighted star graph and also gave some lower bounds on the depth of powers of the integrally closed weighted path.

In this paper,  we are interested in  the depth   of powers of the edge ideal $I(C_{\bf w}^n)$ of a weighted  integrally closed  $n$-cycle $C_{\bf w}^n$.
	The results we obtained are as follows:	

\begin{Theorem}
Let $C_{\bf w}^n$  be a non-trivially weighted integrally closed $n$-cycle. Then we have
\begin{itemize}
		\item[(1)]if $C_{\bf w}^n$ has only one edge with non-trivial weight, then
		 \[
			\depth (S/I(C_{\bf w}^n)^t)=\begin{cases}
			\lceil\frac{n-t}{3}\rceil,& \text{if $1\leq t<\lceil\frac{n+1}{2}\rceil$,}\\
			1,& \text{if $t\geq \lceil\frac{n+1}{2}\rceil$, $n$ is even,}\\
			0,& \text{if $t\geq \lceil\frac{n+1}{2}\rceil$, $n$ is odd.}
			\end{cases}
		\]
		\item[(2)] if $C_{\bf w}^n$ has two edges with non-trivial weights, then
		\[
			\depth (S/I(C_{\bf w}^n)^t)=\max\{\lceil \frac{n-t}{3}\rceil,1\} \text{ for all $t\geq 1$}.
		\]
		\item[(3)]if $C_{\bf w}^n$ has three edges with non-trivial weights, then
		\[
			\depth (S/I(C_{\bf w}^n)^t)=2 \text{ for all $t\geq 1$}.
		\]
\end{itemize}
\end{Theorem}

The article is organized as follows. In Section \ref{sec:prelim}, we provide the definitions and basic facts that will be used throughout this paper.
Sections \ref{sec:cycle} and \ref{sec:power}  are devoted to prove our main results. By the depth lemma, we  provide exact formulas for $\depth(S/I(C_{\bf w}^n))$ for any weighted integrally closed  $n$-cycle $C_{\bf w}^n$, and exact formulas for  $\depth(S/I(C_{\bf w}^n)^t)$ if $C_{\bf w}^n$ has three edges with non-trivial weights or $C_{\bf w}^3$  has two edges with non-trivial weights. If $C_{\bf w}^n$ $(n\ge 4)$ has two edges with non-trivial weights, we can give the lower bound for $\depth(S/I^t)$ by using the depth lemma,
then, by choosing  suitable $f\notin I^t$ and  using the fact that $\depth(S/I(C_{\bf w}^n)^t)\le \depth(S/(I^t:f))$, we prove that $\depth(S/I^t)$ just reaches  this lower bound. If $C_{\bf w}^n$ has only one edge with non-trivial weight, we  can also give  exact formulas for $\depth(S/I(C_{\bf w}^n)^t)$ using  methods from \cite{MTV,ZCLY}.

		\section{Preliminaries}
		\label{sec:prelim}
		In this section, we will provide the definitions and basic facts that  will be used throughout this paper.
		For detailed information, we refer to \cite{BH} and \cite{HH}.
		
			Let $G$ be a  finite simple graph with the vertex set $V(G)$ and the edge set $E(G)$. For any subset $A$ of $V(G)$, $G[A]$ denotes the \emph{induced subgraph} of $G$ on the  set $A$, i.e.,
		for any $u,v \in A$, $uv \in E(G[A])$ if and only if $uv\in E(G)$.
		The induced subgraph of $G$ on the set $V(G)\setminus A$ is denoted by $G\setminus A$. In particular, if $A=\{v\}$  then we will write $G\setminus v$ instead of $G\setminus \{v\}$ for simplicity.
		For a vertex $v\in V(G)$, its \emph{neighborhood} is defined as $N_G(v)\!:=\{u \in V(G)\mid uv\in E(G)\}$, and  its degree  is defined as $\deg_{G}(v)=|N_G(v)|$.
		A connected graph $G$ is called a \emph{cycle} if $\deg_{G}(v)=2$ for all $v\in V(G)$. A cycle with $n$ vertices is called an $n$-cycle, denoted by $C^n$. A connected graph with vertex set $\{x_1,\ldots,x_n\}$, is called a \emph{path}, if $E(G)=\{x_ix_{i+1} \mid 1\leq i\leq n-1\}$. Such a path is  denoted by $P^n$.

		Given a weighted graph $G_{\bf w}$,  we denote its vertex and edge sets by $V(G_{\bf w})$ and $E(G_{\bf w})$, respectively. Any concept valid for graphs automatically applies to weighted graphs.
		For example, the \emph{neighborhood} of a vertex $v$ in a weighted graph $G_{\bf w}$ with the underlying graph  $G$ is defined as  $N_{G_{\bf w}}(v)\!:=\{u \in V(G)\mid uv\in E(G)\}$.
	For a subset  $A$ of $V(G)$, 	the \emph{induced subgraph} of  $G_{\bf w}$ on the set $A$  is the graph $G_{\bf w}[A]$ with vertex set $A$, and for any  $u,v\in A$,
		$uv\in E(G_{\bf w}[A])$ if and only if $uv\in E(G_{\bf w})$, and the weight function ${\bf w}'$ satisfies ${\bf w}'(uv)={\bf w}(uv)$. The induced subgraph of $G_{\bf w}$ on the set $V(G)\setminus A$ is denoted by $G_{\bf w} \setminus A$. In particular, if $A=\{v\}$  then we will write $G_{\bf w} \setminus v$ instead of $G_{\bf w} \setminus \{v\}$ for simplicity.
		A weighted  graph is said to be  a \emph{non-trivially} weighted  graph if there is at least one edge with  a weight  greater than $1$, otherwise, it is said to be
		a {\em trivially} weighted  graph. An edge $e \in E(G_{\bf w})$ is said to be {\em non-trivial} if its weight ${\bf w}(e) \ge2$. Otherwise, it is said to be {\em trivial}.

The following lemmas are often used  to compute the  depth  of a module.

		\begin{Lemma}  {\em (\cite[Lemmas 2.1]{HT})}
			\label{exact}
			Let $0\longrightarrow M\longrightarrow N\longrightarrow P\longrightarrow 0$ be a short exact
			sequence of finitely generated graded S-modules. Then
			\begin{itemize}
				\item[(1)] $\depth(N)\geq \min\{\depth(M), \depth(P)\}$, the equality holds if $\depth(P) \neq \depth(M)-1$;
				\item[(2)] $\depth(M)\geq \min\{\depth(N), \depth(P)+1\}$, the equality holds if $\depth(N) \neq \depth(P)$;
				\item[(3)] $\depth(P)\geq \min\{\depth(M)-1, \depth(N)\}$, the  equality holds if $\depth(M) \neq \depth(N)$.
			\end{itemize}
		\end{Lemma}
		
		\begin{Lemma}{\em (\cite[Lemma 2.2]{HT})}
			\label{sum1}
			Let $S_{1}=\KK[x_{1},\dots,x_{m}]$ and $S_{2}=\KK[x_{m+1},\dots,x_{n}]$ be two polynomial rings  over a field $\KK$. Let $S=S_1\otimes_\KK S_2$ and  $I\subset S_{1}$,
			$J\subset S_{2}$ be two nonzero homogeneous  ideals.  Then
			\begin{itemize}
				\item[(1)] $\depth(S/(I+J))=\depth(S_1/I)+\depth(S_2/J)$;
				\item[(2)] $\depth(S/JI)=\depth(S_1/I)+\depth(S_2/J)+1$.
			\end{itemize}
		\end{Lemma}
	
		\begin{Lemma}{\em (\cite[Corollary 1.3]{R}, \cite[Lemma 2.1]{HT})}\label{up}
		Let $I\subset S$ be a monomial ideal and $f \notin I$ be a monomial of degree $s$. Then
		\[
 \depth( S / I) \leq \depth(S /(I: f)).
\]
Furthermore,  it follows from   Lemma	\ref{exact} and the following short exact sequence
\begin{gather}\label{eqn: equality1}
0  \longrightarrow  \frac{S}{I : f}(-s)   \stackrel{ \cdot f} \longrightarrow   \frac{S}{I}  \longrightarrow  \frac{S}{(I,f)} \longrightarrow  0
	\end{gather}
that
\[
\depth(S/I)\geq \min\{\depth(S/(I:f)), \depth(S/(I,f))\}.
\]
\end{Lemma}

		\begin{Lemma}{\em (\cite[Theorem 3.6]{MTV})}
			\label{path}
			Let $P_{\bf w}^n$ be a trivially weighted path, then  for all $t\ge 1$,
			$$\depth(S/I(P^n_{\bf w})^t)=\max\{\lceil \frac{n-t+1}{3} \rceil,1\}.$$
		\end{Lemma}

	\section{The depth of the edge ideal}
		\label{sec:cycle}
		
		Calculating or even estimating the depth  of edge ideals of weighted cycles with arbitrary weight functions is a  difficult problem. In this section, we will give precise formulas for the  depth of the edge ideal of  a non-trivially weighted  integrally closed  $n$-cycle.  First, we recall that the definition of an ideal is integrally closed.
\begin{Definition}{\em (\cite[Definition 1.4.1]{HH})}
\label{integrally closed_1} Let  $I$  be an ideal in a ring $R$.
    An element $f \in R$ is  said to be {\em integral} over $I$ if there exists an equation
    \[
    f^k+c_1f^{k-1}+\dots+c_{k-1}f+c_k=0 \text{\ \ with\ \ }c_i \in I^i.
    \]
   The set $\overline{I}$ of elements in $R$ which are integral over $I$  is the \emph{integral closure} of $I$.
If $I=\overline{I}$, then $I$  is said to be  {\em integrally closed}.
    A weighted graph $G_\omega$ is said to be  {\em integrally closed} if its edge ideal $I(G_\omega)$ is integrally closed.
    \end{Definition}

We will use the following notation for the remainder of this paper. Let  $n$ be a  positive  integer, the
notation $[n]$ denotes by convention the set $\{1,2,\ldots, n\}$. If $n\ge 3$, let  $C_{\bf w}^n$ denote a non-trivially weighted integrally closed  $n$-cycle with the vertex set  $V(C_{\bf w}^n)=\{x_1,\ldots,x_n\}$ and the edge set $E(C_{\bf w}^n)$. We also denote  $e_i=x_ix_{i+1}\in E(C_{\bf w}^n)$ with $x_{n+1}=x_1$ and ${\bf w}_i={\bf w}(e_i)$ for all $i \in [n]$. By symmetry, we can assume that
${\bf w}_1\geq 2$, ${\bf w}_1\geq {\bf w}_3$ and ${\bf w}_3\geq {\bf w}_5$ if $e_5\in E(C^n_{\bf w})$.
Let $P_{\bf w}^n$ be an induced subgraph of the $(n+1)$-cycle $C_{\bf w}^{n+1}$ on the set $\{x_1,x_2,\ldots,x_n\}$.

According to \cite[Theorem 1.4.6]{HH}, every trivially weighted  graph $G_{\bf w}$ is integrally closed. The following lemma gives a complete characterization of a non-trivially weighted  graph that is integrally closed.
		
		\begin{Lemma}{\em (\cite[Theorem 3.6]{DZCL} and \cite[Theorem 2.8]{VZ}) }\label{integral}
		Let   $G_{\bf w}$ be  a  non-trivially  weighted  graph, then $G_{\bf w}$ is  integrally closed if and only if  $G_{\bf w}$  does not contain  any of the following three graphs as its induced subgraphs.
			\begin{enumerate}%[a]
				\item  A path  $P_{\bf w}^3$ of length $2$ where  all edges have non-trivial weights.
				\item The  disjoint union $P_{\bf w}^2\sqcup P_{\bf w}^2$ of two paths $P_{\bf w}^2$  of length $1$ where all edges have non-trivial weights.
				\item A $3$-cycle $C_{\bf w}^3$ where all edges have non-trivial weights.
			\end{enumerate}
		\end{Lemma}
		
		From the  above lemma, we can derive

\begin{Corollary} \label{cycleinteg}
			Let $C_{\bf w}^n$ be a  non-trivially weighted integrally closed $n$-cycle. If $n=6$, it can have  at most three edges with non-trivial weights, otherwise it can have up to two edges with non-trivial weights.
		\end{Corollary}

		The following lemmas provide some  formulas for the depth of powers of the edge ideal of a non-trivially integrally closed path.

		\begin{Lemma}{\em (\cite[Theorems 4.8 and 4.9]{ZDCL}}{\em )}
			\label{nontrivialpath1}
			Let $P_{\bf w}^n$  be a non-trivially weighted integrally closed  path. Let    ${\bf w}_i\ge 2$ and ${\bf w}_i\ge {\bf w}_{i+2}$ if $e_{i+2}\in E({P_{\bf w}^n})$. Then
			\begin{itemize}
				\item[(1)] if $n\le 3$, then  $\depth (S/I(P_{\bf w}^n))=1$;
				\item[(2)] if $n=4$, then  $\depth (S/I(P_{\bf w}^n))=2-a$, where  $a=\begin{cases}
					0,& \text{if ${\bf w}_2=1$,}\\
					1,& \text{otherwise.}
				\end{cases}$;
				\item[(3)]\label{nontrivialpath1-3}  if $n\ge 5$, then  $\depth (S/I(P_{\bf w}^n))=\min \{\lceil \frac{i}{3} \rceil+\lceil \frac{n-i-b}{3} \rceil,\lceil \frac{i-2}{3} \rceil+\lceil \frac{n-i-2}{3} \rceil+1\}$,
			\end{itemize}
	where  $b=\begin{cases}
					0,& \text{if ${\bf w}_{i+2}=1$,}\\
					1,& \text{otherwise}.
				\end{cases}$.
	\end{Lemma}

		\begin{Lemma}{\em (\cite[Theorems 4.12, 4.17, 4.18 and Proposition 4.13]{ZDCL}}{\em )}
			\label{nontrivialpath2}
			Let $P_{\bf w}^n$  be a non-trivially weighted integrally closed  path as in Lemma \ref{nontrivialpath1}.
Then for all $t\geq 2$, we have the following results:
			\begin{footnotesize}\begin{itemize}
			\item[(1)]  If  $n\le 3$, then  $\depth (S/I(P_{\bf w}^n)^t)=1$;
			\item[(2)]  If  $n=4$, then
			$$
			\depth (S/I(P_{\bf w}^n)^t)\geq
			\begin{cases}
				2,& \text{if ${\bf w}_1,{\bf w}_3>1$ and ${\bf w}_2=1$,}\\
				1,& \text{otherwise.}\\
			\end{cases};
			$$
           \item[(3)]  If $n\geq 5$ and $i=1$, then
			$$
			\depth (S/I(P_{\bf w}^n)^t)\geq
			\begin{cases}
				\max\{\lceil \frac{n-t+1}{3}\rceil,  1\},& \text{if ${\bf w}_3=1$,}\\
				\max\{\lceil \frac{n-t+1}{3}\rceil,  2\},& \text{if ${\bf w}_3>1$.}\\
			\end{cases};
			$$
			\item[(4)]  If $n\geq 5$ and $i>1$, then
			$$
			\depth (S/I(P_{\bf w}^n)^t)\geq
			\begin{cases}
				\max\{\lceil \frac{n-t}{3}\rceil,  2\},& \text{if ${\bf w}_{i+2}>1$,}\\
             \lceil \frac{n-1}{3}\rceil,& \text{if  ${\bf w}_{i+2}=1$ and $i\equiv1(\text{mod}\ 3)$,   $n\equiv2(\text{mod}\ 3)$,    and $t=2$, }\\
				\max\{\lceil \frac{n-t}{3}\rceil,  1\},& \text{otherwise.}\\
			\end{cases}
			$$
		    \end{itemize}
\end{footnotesize}
		\end{Lemma}

		The following lemma provides some exact formulas for the depth of  powers of the edge ideal of a trivially weighted  cycle.
		
	\begin{Lemma}{\em ( }\cite[Proposition  1.3]{C}, \cite[Theorem 1.1]{MTV} and \cite[Theorems 4.4 and 5.1]{T}{\em )}
			\label{trivialcycle}
			Let  $C_{\bf w}^n$ be a trivially weighted  $n$-cycle, then
			$$
			\depth(S/I(C_{\bf w}^n)^t)=
			\begin{cases}\lceil\frac{n-1}{3}\rceil, &\text{if $t=1$,} \\ \lceil\frac{n-t+1}{3}\rceil, &\text {if $2 \leq t<\lceil\frac{n+1}{2}\rceil$,} \\
			1, & \text {if $n$ is even and $t \geq \frac{n}{2}+1$,} \\
			0, & \text {if $n$ is odd and $t \geq \frac{n+1}{2} $.}
			\end{cases}
			$$
		\end{Lemma}

	Next, we will  compute the depth of  the edge ideal of a non-trivially weighted integrally closed $n$-cycle.

		\begin{Theorem}\label{cycle}
 Let $C_{\bf w}^n$ be a non-trivially weighted integrally closed $n$-cycle, then
			$$\depth (S/I(C_{\bf w}^n))=\lceil \frac{n-1}{3} \rceil.$$
		\end{Theorem}
		\begin{proof}
			Let $I=I(C_{\bf w}^n)$ and ${\bf w}_1=\max\{{\bf w}_1,\ldots,{\bf w}_{n}\}$  by symmetry. If $n=3$, then ${\bf w}_2=1$ or ${\bf w}_3=1$ by Lemma \ref{integral}. We can assume ${\bf w}_3=1$ by symmetry. It is easy to see that  $(I:x_3)=(x_2^{{\bf w}_2}x_3^{{\bf w}_2-1},x_1)$ and $(I,x_3)=(x_1^{{\bf w}_1}x_2^{{\bf w}_1},x_3)$. Thus
			$\depth(S/(I:x_3))=\depth(S/(I,x_3))=1$. Applying  Lemma \ref{exact} to the following exact sequence
			\begin{gather*}
					0  \longrightarrow  \frac{S}{I : x_{3}}(-1)   \stackrel{ \cdot x_3} \longrightarrow   \frac{S}{I}  \longrightarrow  \frac{S}{(I,x_{3})}  \longrightarrow  0,
			\end{gather*}
			we can observe that $\depth (S/I)=1$.
			
		In the following, we assume that  $n\geq 4$.  We distinguish between the following two cases:
			
	 {\it Case $1$}:  If $C_{\bf w}^n$ has at most two edges with non-trivial weights, then we can assume  by symmetry and Lemma \ref{integral} that ${\bf w}_1\geq {\bf w}_3$ and ${\bf w}_i=1$ for all $i\neq 1,3$. It is easy to see that  $(I,x_2)=I(C_{\bf w}^{n}\backslash {x_2})+(x_2)$  and $(I:x_2)=I(P_{\bf w}^{n-2})+(x_1^{{\bf w}_1}x_2^{{\bf w}_1-1},x_3)$, where $P_{\bf w}^{n-2}$ is an induced subgraph of $C_{\bf w}^n$ on the set $\{x_1,\widehat{x_2},\widehat{x_3},x_4,\ldots,x_n\}$ and $\widehat{x_2}$ denotes the element $x_2$ omitted from  $\{x_1,x_2,\ldots,x_n\}$. Furthermore, $(I:x_2x_1)=I(P_{\bf w}^{n-4})+(x_1^{{\bf w}_1-1}x_2^{{\bf w}_1-1},x_3,x_n)$ and $((I:x_2),x_1)=I(P_{\bf w}^{n-3})+(x_1,x_3)$, where $P_{\bf w}^{n-4}$ and $P_{\bf w}^{n-3}$ are  induced subgraphs of $C_{\bf w}^n$ on the sets $\{x_4,\ldots,x_{n-1}\}$ and $\{x_4,\ldots,x_n\}$, respectively.
			It follows from Lemmas  \ref{sum1} and \ref{path} that $\depth(S/(I:x_2x_1))=\lceil \frac{n-1}{3} \rceil$ and $\depth(S/((I:x_2),x_1))=\lceil \frac{n}{3} \rceil$. We consider the following two subcases:
			
			(i) If ${\bf w}_3=1$, then $\depth(S/(I,x_2))=\lceil \frac{n-1}{3} \rceil$ by Lemma \ref{path}. It follows that  $\depth (S/I)=\lceil \frac{n-1}{3} \rceil$ by applying Lemma \ref{exact} to
the following short exact sequences
		    \begin{gather}\label{eqn: equality2}
		    	\begin{matrix}
		    		0 & \longrightarrow & \frac{S}{I:x_{2}}(-1)  & \stackrel{ \cdot x_2} \longrightarrow  & \frac{S}{I} & \longrightarrow & \frac{S}{(I,x_{2})} & \longrightarrow & 0,\\
		    		0 & \longrightarrow & \frac{S}{I:x_{2}x_{1}}(-1)  & \stackrel{ \cdot x_1} \longrightarrow  & \frac{S}{I:x_{2}} & \longrightarrow & \frac{S}{((I:x_{2}),x_{1})} & \longrightarrow & 0.
		    	\end{matrix}
		    \end{gather}

		(ii) If ${\bf w}_3\geq 2$, then it follows from Lemmas \ref{sum1} and \ref{nontrivialpath1} that
\[
\depth(S/(I,x_2))=\begin{cases}\lceil\frac{n-1}{3}\rceil, &\text{if $n=4$,} \\
 \lceil\frac{n}{3}\rceil, &\text {if $n\ge5$.}
			\end{cases}
\]
Using Lemma \ref{exact} to the  short exact sequences (\ref{eqn: equality2}),
		    we get  $\depth (S/I)=\lceil \frac{n-1}{3} \rceil$.
		
		 {\it Case $2$}:  If $C_{\bf w}^n$ has three edges with non-trivial weights, then $n=6$ by Corollary \ref{cycleinteg}. In this case, we can assume  by symmetry that ${\bf w}_1\geq {\bf w}_3\geq {\bf w}_5\geq 2$ and ${\bf w}_2={\bf w}_4={\bf w}_6=1$.  It is easy to see that  $(I:x_2)=I(P_{\bf w}^{4})+(x_1^{{\bf w}_1}x_2^{{\bf w}_1-1},x_3)$ and $(I,x_2)=I(P_{\bf w}^{5})+(x_2)$, where $P_{\bf w}^{4}$ and $P_{\bf w}^{5}$ are  induced subgraphs of $C_{\bf w}^6$ on the sets $\{x_1,x_4,x_5,x_6\}$ and $\{x_1,x_3,x_4,x_5,x_6\}$, respectively. Furthermore, $(I:x_2x_1)=(x_1^{{\bf w}_1-1}x_2^{{\bf w}_1-1},x_3,x_4x_5,x_6)$ and $((I:x_2),x_1)=(x_1,x_3,x_4x_5,(x_5x_6)^{{\bf w}_5})$.
		    It follows from Lemmas \ref{sum1}, \ref{path} and \ref{nontrivialpath1} that $\depth(S/(I,x_2))=\depth(S/(I:x_2x_1))=\depth(S/((I:x_2),x_1))=2$.  Applying  Lemma \ref{exact} to the above exact sequences (\ref{eqn: equality2}), we get $\depth (S/I)=2$.
		\end{proof}

		\section{The depth of powers of the edge ideal}
		\label{sec:power}
		In this section, we will give precise formulas for the  depth of  powers of the edge ideal of  a non-trivially integrally closed  $n$-cycle. Thus, in the rest of  this section, we let $C_{\bf w}^n$ be a non-trivially weighted integrally closed  $n$-cycle.

\subsection{$C_{\bf w}^n$ has at least two edges with non-trivial weights}
In this subsection  we  assume that ${\bf w}_1\ge {\bf w}_3\geq 2$.
First, we  consider  the case where $C_{\bf w}^n$ has three edges with non-trivial weights. 	

\begin{Theorem}\label{6-cycle}
			 If $C_{\bf w}^n$ has three edges with non-trivial weights, then
\[
\depth (S/I(C_{\bf w}^n)^t)=2  \text{\ for all\ } t\ge 1.
\]		
\end{Theorem}
	
	    \begin{proof}
	    	Let $I=I(C_{\bf w}^n)$. By  Corollary \ref{cycleinteg}, we get $n=6$.   By symmetry we can assume that ${\bf w}_1\geq {\bf w}_3\geq {\bf w}_5\geq 2$ and ${\bf w}_2={\bf w}_4={\bf w}_6=1$.
Now we prove  the statements by induction on $t$ with the case $t=1$  verified in  Theorem \ref{cycle}.  Thus, in the following, we can assume that $t\ge 2$.
By \cite[Lemmas 4.1 and 4.3]{ZCLY}, we see that $(I^t,x_1)=I(C_{\bf w}^6\backslash x_1)^t+(x_1)$, $(I^t:x_1x_6)=I^{t-1}$, $((I^t:x_{1}),x_{6})=((I^t,x_{6}):x_{1})$, $(I^t,x_6)=I(C_{\bf w}^6\backslash x_6)^t+(x_6)$ and $((I^t,x_6),x_1)=I(P_{\bf w}^4)^t+(x_1,x_6)$, where $P_{\bf w}^{4}$ is the   induced subgraph of $C_{\bf w}^6$ on  $\{x_2,x_3,x_4,x_5\}$. It follows from the induction and Lemmas \ref{sum1} and \ref{nontrivialpath2} that $\depth(S/(I^t,x_1))\ge 2$, $\depth(S/(I^t:x_1x_6))=2$, $\depth(S/(I^t,x_6))\ge 2$ and $\depth(S/((I^t,x_6),x_1))=1$.
	    	Thus, by  Lemma \ref{exact} and  the following short exact sequences
	    	\begin{gather*}
	    		\begin{matrix}
	    			0 & \longrightarrow & \frac{S}{I^t:x_{1}}(-1)  & \stackrel{ \cdot x_1} \longrightarrow  & \frac{S}{I^t} & \longrightarrow & \frac{S}{(I^t,x_{1})} & \longrightarrow & 0,\\
	    			0 & \longrightarrow & \frac{S}{I^t:x_{1}x_{6}}(-1)  & \stackrel{ \cdot x_6} \longrightarrow  & \frac{S}{I^t:x_{1}} & \longrightarrow & \frac{S}{((I^t:x_{1}),x_{6})} & \longrightarrow & 0,\\
	    			0 & \longrightarrow & \frac{S}{((I^t:x_{1}),x_{6})}(-1)  & \stackrel{ \cdot x_1} \longrightarrow  & \frac{S}{(I^t,x_{6})} & \longrightarrow & \frac{S}{((I^t,x_{6}),x_{1})} & \longrightarrow & 0,\\
	    		\end{matrix}
	    	\end{gather*}
	    	we get  $\depth (S/I^t)=2$.
	    \end{proof}

Next, we consider  the case where $C_{\bf w}^n$ has  two  edges with non-trivial weights. First we consider the case where $n=3$.
\begin{Theorem}\label{3-cycle}
		If $C_{\bf w}^3$ has two edges with non-trivial weights, then
\[
\depth(S/I(C_{\bf w}^3)^t)=1  \text{\ for all\ } t\ge 1.
\]		
\end{Theorem}
		  \begin{proof}
	    Let $I=I(C_{\bf w}^3)$, ${\bf w}_1\geq {\bf w}_3\geq 2$ and ${\bf w}_2=1$. We prove  the statements by induction on $t$ with the case $t=1$  verified in  Theorem \ref{cycle}.  Thus, in the following, we can assume that $t\ge 2$.

From the proof of \cite[Theorem 4.2]{Zh1}, we see that
$(I^t: x_2x_3)=I^{t-1}$,  $((I^t : x_2),x_3)=(x_1^{{\bf w}_1t}x_2^{{\bf w}_1t-1},x_3)$,  and $(I^t,x_2)=((x_1x_3)^{{\bf w}_3t},x_2)$. By  induction, we have  $\depth(S/((I^t : x_2),x_3))=\depth(S/(I^t ,x_2))=\depth(S/(I^t : x_2x_3))=1$.
 Using Lemma \ref{exact} to the following short exact sequences
	    \begin{gather*}
	    	\hspace{1cm}\begin{matrix}
	    		0 & \longrightarrow & \frac{S}{I^t : x_2}(-1)  & \stackrel{ \cdot x_2} \longrightarrow  & \frac{S}{I^t} & \longrightarrow & \frac{S}{(I^t,x_2)} & \longrightarrow & 0,\\
	    		0 & \longrightarrow & \frac{S}{I^t : x_2x_3}(-1)  & \stackrel{ \cdot x_3} \longrightarrow  & \frac{S}{I^t : x_2} & \longrightarrow & \frac{S}{((I^t : x_2),x_3)} & \longrightarrow & 0,
	    	\end{matrix}
	    \end{gather*}
	  we can observe that $\depth (S/I^t)=1$.
	    \end{proof}

 Next, we consider the case where  $n\geq 4$.
  \begin{Theorem}\label{geq}
	   	Let $C_{\bf w}^n$ be an  $n$-cycle with $n\geq 4$. If $C_{\bf w}^n$ has two edges with non-trivial weights, then
 \[
\depth (S/I(C_{\bf w}^n)^t)\geq \max\{\lceil \frac{n-t}{3}\rceil,  1\}  \text{\ for all\ } t\ge 1.
\]		
 \end{Theorem}
\begin{proof} Let $I=I(C_{\bf w}^n)$.
	 First,  by  the assumption and Lemma \ref{integral}, we have ${\bf w}_i=1$ for all $i\neq 1,3$.	 We prove  the statements by induction on $t$ with the case $t=1$  verified in  Theorem \ref{cycle}.  Thus, in the following, we can assume that $t\ge 2$.

 By \cite[Lemmas 4.1 and 4.3]{ZCLY}, it  is easy to see that $(I^t,x_3)=I(C_{\bf w}^n\backslash x_3)^t+(x_3)$, $(I^t:x_3x_2)=I^{t-1}$, $((I^t:x_{3}),x_{2})=((I^t,x_{2}):x_{3})$,
 $((I^t,x_2),x_3)=I(P_{\bf w}^{n-2})^t+(x_2,x_3)$, where $P_{\bf w}^{n-2}$ is an induced subgraph of  $C_{\bf w}^n$ on the set $\{x_1,\widehat{x_2},\widehat{x_3},x_4,\ldots,x_n\}$ and $\widehat{x_2}$ denotes the element $x_2$ omitted from  $\{x_1,x_2,\ldots,x_n\}$.
By  induction and Lemmas \ref{sum1} and \ref{nontrivialpath2}, we have
 \begin{align*}
	    		\depth(S/(I^t,x_3))&\geq \max\{\lceil \frac{(n-1)-t+1}{3}\rceil,1\}=\max\{\lceil \frac{n-t}{3}\rceil,1\},\\
	    		\depth(S/((I^t:x_3):x_2))&\geq \max\{\lceil \frac{n-(t-1)}{3}\rceil,1\}=\max\{\lceil \frac{n-t+1}{3}\rceil,1\},\\
	    		\depth(S/(I^t,x_2))&\geq \max\{\lceil \frac{(n-1)-t+1}{3}\rceil,1\}=\max\{\lceil \frac{n-t}{3}\rceil,1\},\\
	    		\depth(S/((I^t,x_2),x_3))&=\geq \max\{\lceil \frac{(n-2)-t+1}{3}\rceil,1\}=\max\{\lceil \frac{n-t-1}{3}\rceil,1\}.
	    	\end{align*}
 Using Lemma \ref{exact} to the following short exact sequences
	    	\begin{gather*}
	    		\begin{matrix}
	    			0 & \longrightarrow & \frac{S}{I^t:x_{3}}(-1)  & \stackrel{ \cdot x_3} \longrightarrow  & \frac{S}{I^t} & \longrightarrow & \frac{S}{(I^t,x_{3})} & \longrightarrow & 0,\\
	    			0 & \longrightarrow & \frac{S}{I^t:x_{3} x_{2}}(-1)  & \stackrel{ \cdot x_2} \longrightarrow  & \frac{S}{I^t:x_{3}} & \longrightarrow & \frac{S}{((I^t:x_{3}),x_{2})} & \longrightarrow & 0,\\
	    			0 & \longrightarrow & \frac{S}{((I^t:x_{3}),x_{2})}(-1)  & \stackrel{ \cdot x_3} \longrightarrow  & \frac{S}{(I^t,x_{2})} & \longrightarrow & \frac{S}{((I^t,x_{2}),x_{3})} & \longrightarrow & 0,\\
	    		\end{matrix}
	    	\end{gather*}
	    	we get  $\depth (S/I^t)\geq \{\max{\lceil \frac{n-t}{3}\rceil,  1}\}$.
	    \end{proof}

	For a  monomial ideal $I\subset S$, let $\mathcal{G}(I)$ denote the unique minimal set of its monomial generators.
\begin{Theorem}\label{4-cycle}
	If $C_{\bf w}^4$ has two edges with non-trivial weights, then
\[
\depth(S/I(C_{\bf w}^4)^t)\le 1  \text{\ for all\ } t\ge 2.
\]		
\end{Theorem}
	\begin{proof}
		Let $I=I(C_{\bf w}^4)$, ${\bf w}_1\geq {\bf w}_3\geq 2$ and ${\bf w}_2={\bf w}_4=1$.
Choose $f=x_1(x_3x_4)^{(t-1){{\bf w}_3}}$, then $(I^t:f)=(x_1,x_2,x_4)$.
In fact,
	\begin{align*}
			x_1f&=x_1^2(x_3x_4)^{(t-1){{\bf w}_3}}=x_3^{{\bf w}_3}x_4^{{\bf w}_3-2}(x_3x_4)^{(t-2){{\bf w}_3}}(x_4x_1)^2\in I^t,\\
			x_2f&=x_1x_2(x_3x_4)^{(t-1){{\bf w}_3}}=(x_3x_4)^{{\bf w}_3-1}(x_2x_3)(x_3x_4)^{(t-2){{\bf w}_3}}(x_4x_1)\in I^t,\\
			x_4f&=x_1x_4(x_3x_4)^{(t-1){{\bf w}_3}}=(x_3x_4)^{(t-1){{\bf w}_3}}(x_4x_1)\in I^t.
		\end{align*}
So $(x_1,x_2,x_4)\subseteq (I^t:f)$.

On the other hand, let $u \in \mathcal{G}(I^t: f)$ be a monomial, then  we can write $uf$ as
	   \begin{align}\label{eqn: equality3}
	   u(x_1(x_3x_4)^{(t-1){{\bf w}_3}})=(x_1x_2)^{k_1{\bf w}_1}(x_2x_3)^{k_2}(x_3x_4)^{k_3{\bf w}_3}(x_4x_1)^{k_4}h,
	   \end{align}
	   where $h$ is a monomial and  $\sum\limits_{i=1}^4k_i=t$ with $k_i\geq 0$ for any $i\in [4]$. If $k_3=t$, then  Eq. (\ref{eqn: equality3}) becomes
	$u(x_1(x_3x_4)^{(t-1){{\bf w}_3}})=(x_3x_4)^{t{\bf w}_3}h$.
It follows that $u\in  (x_4)\subseteq(x_1,x_2,x_4)$. Now, let $0\leq k_3\leq t-1$. Consider the following two cases:
	
	   (1) If $k_1\geq 1$ or $k_2\geq 1$, then  $u\in (x_2)$ by the choice of $f$ and  Eq. (\ref{eqn: equality3}).
	
	   (2) If $k_1=k_2=0$, then  Eq. (\ref{eqn: equality3}) becomes $u(x_1(x_3x_4)^{(t-1){{\bf w}_3}})=(x_3x_4)^{k_3{\bf w}_3}(x_4x_1)^{k_4}h$ with $k_3+k_4=t$. Thus  $u\in (x_4)\subseteq(x_1,x_2,x_4)$ if $k_3=t-1$; otherwise, $u\in (x_1)\subseteq(x_1,x_2,x_4)$. Hence, $(I^t:f)\subseteq (x_1,x_2,x_4)$.

Therefore, $\depth(S/I(C_{\bf w}^4)^t)\le \depth(S/(I(C_{\bf w}^4)^t:f))=1$ by Lemma \ref{up}.
\end{proof}	

A direct consequence follows from  Theorems \ref{cycle}, \ref{geq} and \ref{4-cycle}.
   \begin{Corollary}
    If $C_{\bf w}^4$ has two edges with non-trivial weights, then
    \[
    \depth(S/I(C_{\bf w}^4)^t)=1  \text{\ for all\ } t\ge 1.
    \]				
    \end{Corollary}

\medskip
In the following, we use the notation $u_i=(x_{i}x_{i+1})^{{\bf w}_i}$  where $x_{n+1}=x_1$, and $\prod\limits_{\ell=i}^{j}u_{\ell}=1$ if $i>j$.

 \begin{Theorem}\label{2n-cycle}
      	Let $C_{\bf w}^n$ be an  $n$-cycle with $n\geq 5$. If $C_{\bf w}^n$ has two edges with non-trivial weights.  	Let $t\geq 2$ and choose
    \[
    f=\begin{cases}
    		x_5(x_3x_4)^{{\bf w}_3}\prod\limits_{k=5}^{t+2}(x_kx_{k+1}), &\text{if $2\leq t\leq n-3$,} \\ x_5(x_3x_4)^{(t-n+3){\bf w}_3}\prod\limits_{k=5}^{n}(x_kx_{k+1}), &\text{if $t\geq n-2$.} \\
    	\end{cases}
\]
    	Then
\[
(I(C_{\bf w}^n):f)=
    	\begin{cases}
    		(x_2)+\sum\limits_{k=4}^{t+4}(x_k)+I(C_{\bf w}^n), &\text{if $2\leq t\leq n-3$,} \\
    		\sum\limits_{j \in [n]\setminus\{3\} }(x_j), &\text{if $t\geq n-2$.} \\
    	\end{cases}
\]
    \end{Theorem}
	
	\begin{proof} Let $I=I(C_{\bf w}^n)$ and ${\bf w}_1\geq {\bf w}_3\geq 2$ and  ${\bf w}_i=1$ for all $i\neq 1,3$. Then
 it is trivial that $I(C_{\bf w}^n)\subseteq (I^t:f)$ by the choice of $f$, and $x_2\in (I^t:f)$ if $2\leq t\leq n-3$, since $x_2f=(x_3x_4)^{{{\bf w}_3}-1}u_2\prod\limits_{k=4}^{t+2}u_k\in I^t$.
Now we consider the case where 	 $2\leq t\leq n-3$ and $4\leq i\leq t+4$.	
 In this case, if $i$ is even, then $x_if=u_3(\prod\limits_{\ell=2}^{\frac{i}{2}-1}u_{2\ell+1}^2)(\prod\limits_{\ell=i}^{t+2}u_\ell)\in I^t$.
Otherwise, $x_if=x_3^{{{\bf w}_3}}x_4^{{{\bf w}_3}-2}(\prod\limits_{\ell=2}^{\frac{i-1}{2}}u_{2\ell}^2)(\prod\limits_{\ell=i}^{t+2}u_\ell)\in I^t$.
Therefore, $(x_2)+\sum\limits_{k=4}^{t+4}(x_k)+I(C_{\bf w}^n)\subseteq (I^t:f)$.

On the other hand, let $u \in \mathcal{G}(I^t: f)$ be  a monomial,  then   $uf$ can be written as
\begin{align}\label{eqn: equality4}
			ux_5(x_3x_4)^{{\bf w}_3}\prod\limits_{k=5}^{t+2}(x_kx_{k+1})=h\prod\limits_{i=1}^{n}u_i^{k_i}
		\end{align}
 for some monomial $h$, where $\sum\limits_{i=1}^{n}k_i=t$ and each $k_i\geq 0$.
It is easy to see that if $k_i\geq 1$ for some $i\in\{1,2, t+3\}$ then $u\in (x_2)+(x_{t+4})$, and if $k_i\geq 1$ for some  $\{t+4,\ldots, n\}$  then $u\in(x_ix_{i+1})\subseteq I(C_{\bf w}^n)$.
In the following, we assume that $k_i=0$ for each $i\in \{1,2\}\cup\{t+3,t+4,\ldots, n\}$. In this case, Eq. (\ref{eqn: equality4}) becomes
		\begin{align}\label{eqn: equality5}
			ux_5(x_3x_4)^{{\bf w}_3}\prod\limits_{k=5}^{t+2}(x_kx_{k+1})=h\prod\limits_{i=3}^{t+2}u_i^{k_i},
		\end{align}
		where  $\sum\limits_{i=3}^{t+2}k_i=t$.
We distinguish into three subcases:

(a) If $k_3\geq 2$, then $u\in((x_3x_4)^{{\bf w}_3})\subseteq I(C_{\bf w}^n)$.

(b) If $k_3=1$, then  Eq. (\ref{eqn: equality5}) becomes $ux_5\prod\limits_{k=5}^{t+2}(x_kx_{k+1})=h\prod\limits_{i=4}^{t+2}u_i^{k_i}$ with   $\sum\limits_{i=4}^{t+2}k_i=t-1$. When $k_4\geq 1$,   $u\in (x_4)$. Otherwise,   Eq. (\ref{eqn: equality5}) becomes $ux_5\prod\limits_{k=5}^{t+2}(x_kx_{k+1})=h\prod\limits_{i=5}^{t+2}u_i^{k_i}$ with  $\sum\limits_{i=5}^{t+2}k_i=t-1$. Thus
\begin{align*}\label{eqn: equality6}
u\frac{x_5\prod\limits_{k=5}^{t+2}(x_kx_{k+1})}{\gcd(x_5\prod\limits_{k=5}^{t+2}(x_kx_{k+1}),\prod\limits_{i=5}^{t+2}u_i^{k_i})}
=h\frac{\prod\limits_{i=5}^{t+2}u_i^{k_i}}{\gcd(x_5\prod\limits_{k=5}^{t+2}(x_kx_{k+1}),\prod\limits_{i=5}^{t+2}u_i^{k_i})}.
\end{align*}
Since the degrees of monomials $x_5\prod\limits_{k=5}^{t+2}(x_kx_{k+1})$ and $\prod\limits_{i=5}^{t+2}u_i^{k_i}$ are $2t-3$ and $2t-4$, respectively, we can deduce that
$u\in \sum\limits_{i=5}^{t+3}(x_k)$ by comparing the degree of $x_i$ for any $i\in \{5,6,\ldots,t+3\}$.

(c) If $k_3=0$, then  Eq. (\ref{eqn: equality5}) becomes  $ux_5(x_3x_4)^{{\bf w}_3}\prod\limits_{k=5}^{t+2}(x_kx_{k+1})=h\prod\limits_{i=4}^{t+2}u_i^{k_i}$ with $\sum\limits_{i=4}^{t+2}k_i=t$. In this case, if $k_4\geq 3$, then $u\in (x_5)$ by comparing the degree of $x_5$; otherwise,   $x_3^{{\bf w}_3}x_4^{{\bf w}_3-2}\mid h$. It follows that $ux_4^2x_5\prod\limits_{k=5}^{t+2}(x_kx_{k+1})=(\prod\limits_{i=4}^{t+2}u_i^{k_i})\frac{h}{x_3^{{\bf w}_3}x_4^{{\bf w}_3-2}}$. By similar arguments to (b), we can get that $u\in \sum\limits_{i=4}^{t+3}(x_k)$.

In breif, $(I^t:f)\subseteq (x_2)+\sum\limits_{k=4}^{t+4}(x_k)+I(C_{\bf w}^n)$.

 Now we consider the case where $t\geq n-2$ and  divide into the following four cases:

 {\it Case $1$}:  If both $n$  and $i$ are odd, then $x_if\in I^t$, since
$$
x_if=
\begin{cases}
	u_3^{t-n+3}(\prod\limits_{\ell=2}^{\frac{n-1}{2}}u_{2\ell+1}^2), &\text{if $i=1$,} \\
	x_3^{{\bf w}_3}x_4^{{\bf w}_3-2}u_3^{t-n+2}(\prod\limits_{\ell=2}^{\frac{i-1}{2}}u_{2\ell}^2)(\prod\limits_{\ell=i}^{n}u_{\ell}), &\text{if $i\neq 1,3$.}
\end{cases}
$$

{\it Case $2$}:  If both $n$   and $i$ are  even, then $x_if\in I^t$, since
\[
	x_if=
	\begin{cases}
		(x_3x_4)^{{\bf w}_3-1}u_2u_3^{t-n+2}(\prod\limits_{\ell=4}^{n}u_{\ell}), &\text{if $i=2$,} \\
		u_3^{t-n+3}(\prod\limits_{\ell=2}^{{\frac{i}{2}}-1}u_{2\ell+1}^2)(\prod\limits_{\ell=i}^{n}u_{\ell}), &\text{if $i\neq 2$.}
	\end{cases}
\]
{\it Case $3$}:  If $n$ is  odd   and $i$ is   even, then $x_if\in I^t$, since
\[
x_if=
	\begin{cases}
		(x_3x_4)^{{\bf w}_3-1}u_2u_3^{t-n+2}(\prod\limits_{\ell=4}^{n}u_{\ell}), &\text{if $i=2$,} \\
		u_3^{t-n+3}(\prod\limits_{\ell=2}^{\frac{i}{2}-1}u_{2\ell}^2)(\prod\limits_{\ell=i}^{n}u_{\ell}), &\text{if $i\neq 2$.}
	\end{cases}
\]
{\it Case $4$}:  If $n$ is  even and $i$ is odd, then $x_if\in I^t$, since
$$
x_if=
\begin{cases}
	x_3^{{\bf w}_3}x_4^{{\bf w}_3-2}u_3^{t-n+2}(\prod\limits_{\ell=2}^{\frac{n}{2}}u_{2\ell}^2), &\text{if $i=1$,} \\
x_3^{{\bf w}_3}x_4^{{\bf w}_3-2}u_3^{t-n+2}(\prod\limits_{\ell=2}^{\frac{i-1}{2}}u_{2\ell}^2)(\prod\limits_{\ell=i}^{n}u_{\ell}), &\text{if $i\neq 1,3$.}
\end{cases}
$$

In brief, $\sum\limits_{i \in [n]\setminus\{3\} }(x_i)\subseteq (I^t:f)$.

Next, we will prove that $(I^t:f)\subseteq \sum\limits_{i \in [n]\setminus\{3\} }(x_i)$
 by induction on $t$.

(a) If $t=n-2$, then $f=x_5(x_3x_4)^{{\bf w}_3}\prod\limits_{k=5}^{n}(x_kx_{k+1})$. Let  $u \in \mathcal{G}(I^t: f)$ be a monomial, then we can write $uf$ as
 \begin{align}\label{eqn: equality7}
	ux_5(x_3x_4)^{{\bf w}_3}\prod\limits_{k=5}^{n}(x_kx_{k+1})=h\prod\limits_{i=1}^{n}u_i^{k_i}
\end{align}
for some monomial $h$, where $\sum\limits_{i=1}^{n}k_i=t$ and each $k_i\geq 0$.

If $k_1\geq 1$ or $k_2\geq 1$, then  $u\in (x_2)$. In the following, we will assume that $k_1=k_2=0$. In this case, Eq.  (\ref{eqn: equality7}) becomes
\begin{align}\label{eqn: equality8}
	ux_5(x_3x_4)^{{\bf w}_3}\prod\limits_{k=5}^{n}(x_kx_{k+1})=h\prod\limits_{i=3}^{n}u_i^{k_i}
\end{align}
with $\sum\limits_{i=3}^{n}k_i=t$. We consider the following three subcases:

(i) If $k_3\geq 2$, then $u\in (x_4)$ by comparing the degree of $x_4$  in Eq.  (\ref{eqn: equality8}).

(ii) If $k_3=1$,  then   Eq.  (\ref{eqn: equality8}) becomes
	$ux_5\prod\limits_{k=5}^{n}(x_kx_{k+1})=h\prod\limits_{i=4}^{n}u_i^{k_i}$
with $\sum\limits_{i=4}^{n}k_i=t-1$. Further,
if $k_4\geq 1$, then $u\in (x_4)$ by  comparing the degree of $x_4$; otherwise, Eq.  (\ref{eqn: equality8}) becomes
$ux_5\prod\limits_{k=5}^{n}(x_kx_{k+1})=h\prod\limits_{i=5}^{n}u_i^{k_i}$
with $\sum\limits_{i=5}^{n}k_i=t-1$. Similar arguments as in (1)(b), we can get $u\in(x_1)+\sum\limits_{i=5}^{n}(x_k)$.

(iii) If $k_3=0$, then  Eq.  (\ref{eqn: equality8}) becomes	$ux_5(x_3x_4)^{{\bf w}_3}\prod\limits_{k=5}^{n}(x_kx_{k+1})=h\prod\limits_{i=4}^{n}u_i^{k_i}$
with $\sum\limits_{i=4}^{n}k_i=t$. Furthermore, if $k_4\geq 3$, then  $u\in (x_5)$  by  comparing the degree of $x_5$; otherwise, by similar arguments as in (1)(c),
we can see that $u\in(x_1)+\sum\limits_{k=4}^{n}(x_k)$.

(b) If $t\geq n-1$. Let $u \in \mathcal{G}(I^t: f)$ be  a monomial, then  we can write $uf$ as
$ux_5(x_3x_4)^{(t-n+3){\bf w}_3}\prod\limits_{k=5}^{n}(x_kx_{k+1})=h\prod\limits_{i=1}^{n}u_i^{k_i}$
for some monomial $h$, where $\sum\limits_{i=1}^{n}k_i=t$ and each $k_i\geq 0$. We consider the following two subcases:

(i) If $k_3\geq 1$, then  $u\frac{f}{(x_3x_4)^{{\bf w}_3}}\in I^{t-1}$.  From   induction it follows that
 $(I^t:f)\subseteq (I^{t-1}:\frac{f}{(x_3x_4)^{{\bf w}_3}})\subseteq \sum\limits_{i \in [n]\setminus\{3\} }(x_i)$.

(ii) If $k_3=0$. We consider two cases: (i) if $k_1\geq 1$ or $k_2\geq 1$, then  $u\in(x_2)$  by  comparing the degree of $x_2$; (ii) if $k_1=k_2=0$, then $u\in (x_1)+\sum\limits_{k=4}^{n}(x_k)$ by similar arguments as in  (2)(a)(iii).
	
We complete	the proof.
\end{proof}

 \begin{Theorem}
		Let $C_{\bf w}^n$ be an $n$-cycle with $n\geq 4$. If $C_{\bf w}^n$ has two edges with non-trivial weights, then $\depth (S/I(C_{\bf w}^n)^t)= \max\{\lceil \frac{n-t}{3}\rceil,1\}$ for all $t\ge 1$.
	\end{Theorem}
		\begin{proof}
By Theorem \ref{geq},  $\depth (S/I(C_{\bf w}^n)^t)\geq  \{\max{\lceil \frac{n-t}{3}\rceil, 1}\}$. On the other hand, by Lemmas \ref{up} and \ref{path}, Theorems \ref{4-cycle} and \ref{2n-cycle}, we can obtain that   $\depth (S/I(C_{\bf w}^n)^t)\leq  \{\max{\lceil \frac{n-t}{3}\rceil,1}\}$.
	\end{proof}

\subsection{$C_{\bf w}^n$ has only one edge with non-trivial weights}
In this section, we  suppose ${\bf w}_1\geq 2$ and ${\bf w}_i=1$ for all $i\neq 1$.
We need the following lemmas.
	
	\begin{Lemma}{\em (\cite[Lemma 3.4]{MTV})}\label{path3}
		Let $P_{\bf w}^n$  be a trivially weighted path and  $H$  its subgraph. Then
	$\depth\left(S/(I(P_{\bf w}^n)^t+I(H))\right)\geq \lceil \frac{n-t+1}{3}\rceil$   for any $t<n$.
	\end{Lemma}

		\begin{Lemma}{\em (\cite[Lemma 4.10]{ZDCL}}{\em )}\label{path1}
			Let $P_{\bf w}^n$ be a weighted  path and  ${\bf w}_{n-1}=1$. Then for all $t \geq 2$, we have
			\begin{itemize}
			\item[(1)] $(I(P_{\bf w}^n)^t: x_{n-1} x_n)=I(P_{\bf w}^n)^{t-1}$;
	    	\item[(2)] $((I(P_{\bf w}^n)^t: x_n), x_{n-1})=(I(P_{\bf w}^n \backslash x_{n-1})^t, x_{n-1})$;
			\item[(3)] $(I(P_{\bf w}^n)^t, x_n)=(I(P_{\bf w}^n \backslash x_n)^t, x_n)$;
			\item[(4)] $(I(P_{\bf w}^n)^t, x_{n-1})=(I(P_{\bf w}^n \backslash x_{n-1})^t, x_{n-1})$;
			\item[(5)] $((I(P_{\bf w}^n)^t: x_{n-1}), x_n)=((I(P_{\bf w}^n \backslash x_n)^t: x_{n-1}), x_n)$.
			\end{itemize}
		\end{Lemma}

  \begin{Theorem}\label{pathsubgraph}
    	Let $P_{\bf w}^n$  be a path with only  one edge of non-trivial weight. Let ${\bf w}_i\ge 2$, ${\bf w}_j=1$ for all $j\neq i$, and let $P_{\bf w}^{n-i+1}$ be the induced subgraph of $P_{\bf w}^n$ on the set $\{x_i,x_{i+1},\ldots,x_n\}$ and $H_{\bf w}$ be any subgraph of $P_{\bf w}^{n-i+1}$. Then for any $1\leq t<\lceil\frac{n+1}{2}\rceil$, we have
    	\begin{itemize}
    		\item[(1)] $\depth\left(S/(I(P_{\bf w}^n)^t+I(H_{\bf w}))\right)\geq
    		\begin{cases}
    			\lceil \frac{n-t+1}{3}\rceil, &\text{if $i=1$,} \\ 	
    			\lceil \frac{n-t}{3}\rceil, &\text{if $i\geq 2$.} \\ 	
    		\end{cases} $
    		\item[(2)] If $i\geq 2$ and  $P_{\bf w}^{i}$ is the induced subgraph of $P_{\bf w}^n$ on  $\{x_1,x_{2},\ldots,x_i\}$, then
    		$$
    		\depth\left(S/(I(P_{\bf w}^n)^t+I(H_{\bf w})+I(P_{\bf w}^{i}))\right)\geq \lceil \frac{n-t}{3}\rceil.
    		$$
    	\end{itemize}
    \end{Theorem}
\begin{proof}
Let $I=I(P_{\bf w}^n)$. We first prove the statement (1) by induction on $t$. % $n$ and $i$.
    	
     If $t=1$, then $I(P_{\bf w}^n)^t+I(H_{\bf w})=I(P_{\bf w}^n)$. By Lemma \ref{nontrivialpath1}, we can get $$\depth (S/I(P_{\bf w}^n))\ge
    	\begin{cases}
    		\lceil \frac{n}{3}\rceil, &\text{if $i=1$,} \\ 	
    		\lceil \frac{n-1}{3}\rceil, &\text{if $i\geq 2$.} \\ 	
    	\end{cases} $$
    In the following, we can assume that  $t\geq 2$. Note that $\sqrt{I(P_{\bf w}^n)^t+I(H_{\bf w})}=\sqrt{I(P_{\bf w}^n)}$,
    we have  $\mathfrak{m}\notin Ass_S(I(P_{\bf w}^n)^t+I(H_{\bf w}))$ for all $i$, $t$ and $H_{\bf w}$, where $Ass_S(I(P_{\bf w}^n)^t+I(H_{\bf w}))$ is the set of  associated prime ideals of the $S$-module $S/(I(P_{\bf w}^n)^t+I(H_{\bf w}))$. So $\depth (S/(I(P_{\bf w}^n)^t+I(H_{\bf w})))\geq 1\geq  \lceil \frac{n-t+1}{3}\rceil$ if $n\leq 4$.

Now we assume that   $n\geq 5$ and $A=I(P_{\bf w}^n)^t+I(H_{\bf w})$.
First, we consider the case where  $i=1$ and divide into the following four cases:

       {\it Case $1$}:  If $x_{n-2}x_{n-1}, x_{n-1}x_{n}\notin E(H_{\bf w})$, then, by Lemma \ref{path1}, we see that
        \begin{align*}
        	(A,x_{n-1})&=I(P_{\bf w}^n\backslash x_{n-1})^t+I(H_{\bf w})+(x_{n-1}),\\
        	(A:x_{n-1}x_n)&=I(P_{\bf w}^n)^{t-1}+I(H_{\bf w}),\\
        ((A:x_{n-1}),x_n)&=(I(P_{\bf w}^n\backslash x_n)^{t}:x_{n-1})+I(H_{\bf w})+(x_n),\\
        (((A:x_{n-1}),x_n),x_{n-2})&=I(P_{\bf w}^n\backslash \{x_{n-2},x_n\})^{t}+I(H_{\bf w}\backslash x_{n-2})+(x_{n-2},x_n),\\
        	(((A:x_{n-1}),x_n):x_{n-2})&=
        	\begin{cases}
        		I(P_{\bf w}^n\backslash x_n)^{t-1}+I(H_{\bf w})+(x_n),  \text{if $x_{n-3}x_{n-2}\notin E(H_{\bf w})$,} \\ 	
        		B,\ \  \text{otherwise,} 	
        	\end{cases}
        \end{align*}
      where $B=I(P_{\bf w}^n\backslash \{x_{n-3},x_n\})^{t-1}+I(H_{\bf w}\backslash x_{n-3})+(x_{n-3},x_n)$.  Therefore, by  induction and   Lemma \ref{sum1}, we obtain that
        \begin{align*}
        	\depth(S/(A,x_{n-1}))&\geq \lceil \frac{(n-2)-t+1}{3}\rceil+1=\lceil \frac{n-t+2}{3}\rceil,\\
        	\depth(S/(A:x_{n-1}x_n))&\geq \lceil \frac{n-(t-1)+1}{3}\rceil=\lceil \frac{n-t+2}{3}\rceil,\\
        	\depth(S/(((A:x_{n-1}),x_n),x_{n-2}))&\geq \lceil \frac{(n-3)-t+1}{3}\rceil+1=\lceil \frac{n-t+1}{3}\rceil.
        \end{align*}
       Furthermore,   if $x_{n-3}x_{n-2}\notin E(H_{\bf w})$, then $\depth(S/(((A:x_{n-1}),x_n):x_{n-2}))\geq  \lceil \frac{(n-1)-(t-1)+1}{3}\rceil=\lceil \frac{n-t+1}{3}\rceil$;
       otherwise, we let $K_\ell=I(P_{\bf w}^n\backslash \{x_{n-3},x_n\})^{\ell}+I(H_{\bf w}\backslash x_{n-3})+(x_{n-3},x_n)$  for all $\ell\in  [t-1]$, then $K_1=I(P_{\bf w}^n\backslash \{x_{n-3},x_n\})+I(H_{\bf w}\backslash x_{n-3})+(x_{n-3},x_n)$,
$(K_\ell,x_{n-2}x_{n-1})=I(P_{\bf w}^n\backslash \{x_{n-3},x_{n-2},x_{n-1},x_n\})^{\ell}+I(H_{\bf w}\backslash x_{n-3})+(x_{n-3},x_{n-2}x_{n-1},x_n)$ for all $\ell\in [t-1]$,  $(K_\ell:x_{n-2}x_{n-1})=K_{\ell-1}$ for all $2\leq \ell\leq t-1$,
and $K_{t-1}=(((A:x_{n-1}),x_n):x_{n-2})$.
By induction, Lemmas \ref{sum1} and \ref{nontrivialpath1}, we get that $\depth(S/K_1)\ge \lceil \frac{n-4}{3}\rceil+1=\lceil \frac{n-1}{3}\rceil$ and
$\depth(S/(K_\ell,x_{n-2}x_{n-1}))\geq \lceil \frac{(n-4)-(t-1)+1}{3}\rceil+1=\lceil \frac{n-t+1}{3}\rceil$ for any $2\leq \ell\leq t-1$.
Applying  Lemma \ref{exact} to the following short exact sequences
\begin{gather*}
       	\begin{matrix}
       		0 & \longrightarrow & \frac{S}{K_{t-2}}(-2) & \xrightarrow[]{\cdot x_{n-2}x_{n-1}} & \frac{S}{K_{t-1}} & \longrightarrow & \frac{S}{(K_{t-1},x_{n-2}x_{n-1})} & \longrightarrow & 0,\\
       		0 & \longrightarrow & \frac{S}{K_{t-3}}(-2) & \xrightarrow[]{\cdot x_{n-2}x_{n-1}} & \frac{S}{K_{t-2}} & \longrightarrow & \frac{S}{(K_{t-2},x_{n-2}x_{n-1})} & \longrightarrow & 0,\\
       	   	&  &\vdots&  &\vdots&  &\vdots&  &\\
       		0 & \longrightarrow & \frac{S}{K_1}(-2) & \xrightarrow[]{\cdot x_{n-2}x_{n-1}} & \frac{S}{K_2} & \longrightarrow & \frac{S}{(K_2,x_{n-2}x_{n-1})} & \longrightarrow & 0,\\
       	\end{matrix}
       \end{gather*}
we deduce  $\depth(S/K_{t-1})\geq \lceil\frac{n-t+1}{3}\rceil$. In both cases, we always have $\depth(S/K_{t-1})\geq \lceil\frac{n-t+1}{3}\rceil$.

 Using Lemma \ref{exact} to the  following short exact sequences
   \begin{gather}\label{eqn: equality10}
       	\begin{matrix}
       	0 \!&\!\longrightarrow \!&\! \frac{S}{A:x_{n-1}}(-1) \!&\! \stackrel{ \cdot x_{n-1}}\longrightarrow \!&\!\frac{S}{A} \!&\!\longrightarrow\!&\! \frac{S}{(A,x_{n-1})} \!&\! \longrightarrow \!&\! 0,\\
       	0 \!&\! \longrightarrow \!&\! \frac{S}{A:x_{n-1}x_{n}}(-1) \!&\! \stackrel{ \cdot x_n}\longrightarrow \!&\!\frac{S}{A:x_{n-1}}\!&\!\longrightarrow \!&\! \frac{S}{((A:x_{n-1}),x_{n})} \!&\! \longrightarrow \!&\! 0,\\
       	0 \!&\! \longrightarrow \!&\! \frac{S}{((A:x_{n-1}),x_{n}):x_{n-2}}(-1) \!&\! \stackrel{ \cdot x_{n-2}}\longrightarrow  \!&\! \frac{S}{((A:x_{n-1}),x_{n})} \!&\! \longrightarrow\!&\! \frac{S}{(((A:x_{n-1}),x_{n}),x_{n-2})} \!&\! \longrightarrow \!&\! 0,
       	\end{matrix}
       \end{gather}	
we obtain that $\depth (S/A)\geq \lceil\frac{n-t+1}{3}\rceil$.

 {\it Case $2$}:  If $x_{n-2}x_{n-1}, x_{n-1}x_{n}\in E(H_{\bf w})$, then by Lemma \ref{path1}, we have
\begin{align*}
	(A,x_{n-1})&=I(P_{\bf w}^n\backslash x_{n-1})^t+I(H_{\bf w}\backslash x_{n-1})+(x_{n-1}),\\
	(A:x_{n-1})&=I(P_{\bf w}^n\backslash \{x_{n-2},x_n\})^{t}+I(H_{\bf w}\backslash \{x_{n-2},x_n\})+(x_{n-2},x_n).
\end{align*}
By induction  and Lemma \ref{sum1}, we get  that
\begin{align*}
	\depth(S/(A,x_{n-1}))&\geq\lceil \frac{(n-2)-t+1}{3}\rceil+1=\lceil \frac{n-t+2}{3}\rceil,\\
	\depth(S/(A:x_{n-1}))&\geq \lceil \frac{(n-3)-t+1}{3}\rceil+1=\lceil \frac{n-t+1}{3}\rceil.
\end{align*}
It follows from Lemma \ref{exact} and the following exact sequence
\begin{gather*}
	\begin{matrix}
		0&\longrightarrow&\frac{S}{A:x_{n-1}}(-1)&\stackrel{ \cdot x_{n-1}} \longrightarrow&\frac{S}{A}&\longrightarrow&\frac{S}{(A,x_{n-1})}&\longrightarrow&0
	\end{matrix}
\end{gather*}
that $\depth (S/A)\geq \lceil\frac{n-t+1}{3}\rceil$.

 {\it Case $3$}:  If $x_{n-2}x_{n-1}\in E(H_{\bf w})$ and $x_{n-1}x_{n}\notin E(H_{\bf w})$, then by Lemma \ref{path1}, we have
     \begin{align*}
     	(A,x_{n-1})&=I(P_{\bf w}^n\backslash x_{n-1})^t+I(H_{\bf w}\backslash x_{n-1})+(x_{n-1}),\\
     	(A:x_{n-1}x_n)&=I(P_{\bf w}^n\backslash \{x_{n-2}\})^{t-1}+I(H_{\bf w}\backslash x_{n-2})+(x_{n-2}),\\
     	((A:x_{n-1}),x_n)&=(I(P_{\bf w}^n\backslash \{x_{n-2},x_n\})^{t})+I(H_{\bf w}\backslash x_{n-2})+(x_{n-2},x_n).
     \end{align*}
   Therefore, by  induction and  Lemma \ref{sum1},  we can deduce
     that
     \begin{align*}
     	\depth(S/(A,x_{n-1}))&\geq \lceil \frac{(n-2)-t+1}{3}\rceil+1=\lceil \frac{n-t+2}{3}\rceil,\\
     	\depth(S/((A:x_{n-1}),x_n))&\geq \lceil \frac{(n-3)-t+1}{3}\rceil+1=\lceil \frac{n-t+1}{3}\rceil.
     \end{align*}
Similarly to the calculation of the depth of $S/K_{t-1}$ in {\it Case $1$},  we obtain $\depth(S/(A:x_{n-1}x_n))\geq \lceil \frac{n-t+1}{3}\rceil$. Applying Lemma \ref{exact} to the first two exact sequences of the sequences (\ref{eqn: equality10}),  we can get $\depth (S/A)\geq \lceil\frac{n-t+1}{3}\rceil$.

 {\it Case $4$}:  If $x_{n-2}x_{n-1}\notin E(H_{\bf w})$ and $x_{n-1}x_{n}\in E(H_{\bf w})$, then by Lemma \ref{path1}, we have
    \begin{align*}
    	(A,x_{n-1})&=I(P_{\bf w}^n\backslash x_{n-1})^t+I(H_{\bf w}\backslash x_{n-1})+(x_{n-1}),\\
    	((A:x_{n-1}),x_{n-2})&=(I(P_{\bf w}^n\backslash \{x_{n-2},x_n\})^{t})+I(H_{\bf w}\backslash \{x_{n-2},x_n\})+(x_{n-2},x_n),\\
    (A:x_{n-1}x_{n-2})&=
    	\begin{cases}
I(P_{\bf w}^n\backslash x_{n})^{t-1}+I(H_{\bf w}\backslash x_{n})+(x_{n}), \ \text{if $x_{n-3}x_{n-2}\notin E(H_{\bf w})$,} \\ 	
C, \ \ \text{otherwise,} \\ 	
        \end{cases}
    \end{align*}
  where $C=I(P_{\bf w}^n\backslash \{x_{n-3},x_n\})^{t-1}+I(H_{\bf w}\backslash \{x_{n-3},x_n\})+(x_{n-3},x_n)$.

Therefore, from  induction and Lemma \ref{sum1}, we can deduce  that
\begin{align*}
	\depth(S/(A,x_{n-1}))&\geq \lceil \frac{(n-2)-t+1}{3}\rceil+1=\lceil \frac{n-t+2}{3}\rceil,\\
	\depth(S/((A:x_{n-1}),x_{n-2}))&\geq \lceil \frac{(n-3)-t+1}{3}\rceil+1=\lceil \frac{n-t+1}{3}\rceil.
\end{align*}
Similarly to the calculation of the depth of $S/(((A:x_{n-1}),x_n):x_{n-2})$ in {\it Case $1$},
we can get $\depth(S/(A:x_{n-1}x_{n-2}))\geq \lceil \frac{n-t+1}{3}\rceil$.
Again, using Lemma \ref{exact} to the first two exact sequences of the sequence (\ref{eqn: equality10}),  we can get
 $\depth (S/A)\geq \lceil\frac{n-t+1}{3}\rceil$.

Now,  we consider the case where $i\ge 2$. If  $n\leq 5$ and $t\geq 2$, or   $i=3$, $n\leq 6$ and $t\geq 3$, then
 $\mathfrak{m}\notin Ass_S(I(P_{\bf w}^n)^t+I(H_{\bf w}))$ for all $i$, $t$ and $H_{\bf w}$. So $\depth (S/(I(P_{\bf w}^n)^t+I(H_{\bf w})))\geq 1\geq  \lceil \frac{n-t}{3}\rceil$.
  Therefore, in the following, we consider the four remaining cases: (I) $i=2$, $n\geq 6$ and  $t\geq 2$; (II) $i=3$, $n=6$ and  $t=2$; (III) $i=3$, $n\geq 7$ and  $t\geq 2$;  (IV) $i\ge 4$,   $n\ge 6$ and  $t\ge 2$.

First, if the case (I) or (III) occurs, then, by similar arguments as in the
case where $i=1$, we can deduce
 $\depth (S/A)\geq \lceil \frac{n-t}{3}\rceil$.

If the case  (II) occurs,  then it remains to show that $\depth (S/A)\geq 2$. If $x_{4}x_{5}\in E(H_{\bf w})$, then $\depth (S/A)\geq 2$ by similar arguments as in  {\it Case $2$} and {\it Case $3$} for $i=1$. Otherwise, there are two cases to consider:

(a) If $x_{5}x_{6}\in E(H_{\bf w})$, then by Lemma \ref{path1}, we have
    	\begin{align*}
    		(A,x_{5})&=I(P_{\bf w}^6\backslash x_{5})^2+I(H_{\bf w}\backslash x_5)+(x_{5}),\\
    		(A:x_{5}x_4)&=\begin{cases}
    			I(P_{\bf w}^6\backslash x_{6})+(x_{6}), &\text{if $x_{3}x_{4}\notin E(H_{\bf w})$,} \\ 	
    			I(P_{\bf w}^6\backslash x_6)+(x_{3}^{\bf w}x_4^{{\bf w}-1},x_6), &\text{if $x_{3}x_{4}\in E(H_{\bf w})$,} \\ 	
    		\end{cases}\\
    		((A:x_{5}),x_4)&=I(P_{\bf w}^6\backslash \{x_4,x_6\})^{2}+(x_4,x_6).
    	\end{align*}
  By  the inductive hypothesis and Lemma \ref{sum1}, we get that  $\depth(S/((A:x_5),x_4))\geq 2$ and $\depth(S/(A,x_5))\geq 2$.
  To compute  depth of $S/A$, we need to compute  $\depth(S/(A:x_{5}x_4))$ and  distinguish between two subcases: (i) if  $x_{3}x_{4}\notin E(H_{\bf w})$, then $\depth(S/(A:x_{5}x_4))\geq 2$
  by Lemma \ref{nontrivialpath1}(3).  (ii) If  $x_{3}x_{4}\in E(H_{\bf w})$, then $(A:x_{5}x_4x_3)=(x_2,(x_3x_4)^{{\bf w}-1},x_4x_5,x_6)$ and $((A:x_{5}x_4),x_3)=I(P_{\bf w}^6\backslash \{x_3,x_6\})+(x_3,x_6)$ by Lemma \ref{path1}. It follows
  from Lemmas \ref{sum1} and \ref{nontrivialpath1} that $\depth(S/(A:x_{5}x_4x_3))=2$ and $\depth(S/((A:x_{5}x_4),x_3))=2$.
  Applying  Lemma \ref{exact} to the  following exact sequences
  \begin{gather}\label{eqn: equality11}
	\begin{matrix}
		0 & \longrightarrow & \frac{S}{A:x_{5}x_4x_3}(-1)  & \stackrel{ \cdot x_3} \longrightarrow  & \frac{S}{A:x_{5}x_4} & \longrightarrow & \frac{S}{((A:x_{5}x_4),x_3)} & \longrightarrow & 0,\\
	0 & \longrightarrow & \frac{S}{A:x_{5}x_{4}}(-1)  & \stackrel{ \cdot x_4} \longrightarrow  & \frac{S}{A:x_{5}} & \longrightarrow & \frac{S}{((A:x_{5}),x_{4})} & \longrightarrow & 0,\\
0 & \longrightarrow & \frac{S}{A:x_{5}}(-1)  & \stackrel{ \cdot x_5} \longrightarrow  & \frac{S}{A} & \longrightarrow & \frac{S}{(A,x_{5})} & \longrightarrow & 0,
\end{matrix}
\end{gather}
we  deduce that $\depth (S/A)\geq 2$.

(b)  If $x_{5}x_{6}\notin E(H_{\bf w})$, then $(A:x_{5}x_6)=I(P_{\bf w}^6)+I(H_{\bf w})$ and $(A,x_{5}x_6)=I(P_{\bf w}^6)^{2}+I(H_{\bf w})+(x_5x_6)$ by Lemma \ref{path1}. It follows  from Lemma \ref{nontrivialpath1} and the above (a) that
$\depth(S/(A:x_{5}x_6))\geq 2$ and $\depth(S/(A,x_{5}x_6))\geq 2$. This forces $\depth (S/A)\geq 2$ by using Lemma \ref{exact} to the following short exact sequence
\begin{gather*}\label{eqn: equality28}
	\begin{matrix}
		0 & \longrightarrow & \frac{S}{A:x_{5}x_6}(-2)  & \xrightarrow[]{\cdot x_5x_6}  & \frac{S}{A} & \longrightarrow & \frac{S}{(A,x_{5}x_6)} & \longrightarrow & 0.
	\end{matrix}
\end{gather*}

If the case (IV) occurs,  then
\begin{align*}
	(A,x_{2})&=I(P_{\bf w}^n\backslash x_{2})^t+I(H_{\bf w})+(x_{2}),\\
	(A:x_{2}x_1)&=I(P_{\bf w}^n)^{t-1}+I(H_{\bf w}),\\
	((A:x_{2}),x_1)&=(I(P_{\bf w}^n\backslash x_1)^{t}:x_2)+I(H_{\bf w})+(x_1),\\
	(((A:x_{2}),x_1):x_3)&=I(P_{\bf w}^n\backslash x_1)^{t-1}+I(H_{\bf w})+(x_1),\\
	(((A:x_{2}),x_1),x_3)&=I(P_{\bf w}^n\backslash \{x_1,x_3\})^{t}+I(H_{\bf w})+(x_1,x_3).
\end{align*}
By  induction and Lemma \ref{sum1}, we can deduce  that
\begin{align*}
	\depth(S/(A,x_2))&\geq \lceil \frac{(n-2)-t}{3}\rceil+1=\lceil \frac{n-t+1}{3}\rceil,\\
	\depth(S/(A:x_{2}x_1))&\geq \lceil \frac{n-(t-1)}{3}\rceil=\lceil \frac{n-t+1}{3}\rceil,\\
	\depth(S/(((A:x_{2}),x_1):x_3))&\geq \lceil \frac{(n-1)-(t-1)}{3}\rceil=\lceil \frac{n-t}{3}\rceil,\\
	\depth(S/(((A:x_{2}),x_1),x_3))&\geq \lceil \frac{(n-3)-t}{3}\rceil+1=\lceil \frac{n-t}{3}\rceil.
\end{align*}
By  Lemma \ref{exact} and  the following short exact sequences
\begin{gather*}
	\begin{matrix}
		0 & \longrightarrow & \frac{S}{A:x_{2}}(-1)  & \stackrel{ \cdot x_2} \longrightarrow  & \frac{S}{A} & \longrightarrow & \frac{S}{(A,x_{2})} & \longrightarrow & 0,\\
		0 & \longrightarrow & \frac{S}{A:x_{2} x_{1}}(-1)  & \stackrel{ \cdot x_1} \longrightarrow  & \frac{S}{A:x_{2}} & \longrightarrow & \frac{S}{((A:x_{2}),x_{1})} & \longrightarrow & 0,\\
		0 & \longrightarrow & \frac{S}{((A:x_{2}),x_{1}):x_3}(-1)  & \stackrel{ \cdot x_3} \longrightarrow  & \frac{S}{((A:x_{2}),x_{1})} & \longrightarrow & \frac{S}{(((A:x_{2}),x_{1}),x_{3})} & \longrightarrow & 0,\\
	\end{matrix}
\end{gather*}
we can get $\depth (S/A)\geq \lceil\frac{n-t}{3}\rceil$.

\medskip
Final, we will prove the statement (2) by induction on $t$, its proof  is similar as in the previous case.

 If $t=1$, then $I(P_{\bf w}^n)^t+I(P_{\bf w}^i)+I(H_{\bf w})=I(P_{\bf w}^n)$.  So $\depth (S/I(P_{\bf w}^n))=\lceil \frac{n-1}{3}\rceil$ by Lemma \ref{nontrivialpath1}, since $i\geq 2$.

Next, we assume that  $t\geq 2$. Since
 $\sqrt{I(P_{\bf w}^n)^t+I(P_{\bf w}^i)+I(H_{\bf w})}=\sqrt{I(P_{\bf w}^n)}$, $\mathfrak{m}\notin Ass_S(I(P_{\bf w}^n)^t+I(P_{\bf w}^i)+I(H_{\bf w}))$ for all $i$, $t$ and $H_{\bf w}$. So $\depth (S/(I(P_{\bf w}^n)^t+I(P_{\bf w}^i)+I(H_{\bf w}))\geq 1\geq  \lceil \frac{n-t}{3}\rceil$ if $n\leq 5$.

In the following, we assume that   $n\geq 6$ and $D=I(P_{\bf w}^n)^t+I(P_{\bf w}^i)+I(H_{\bf w})$.

(i) If $i=2$, then $(D:x_1)=I(P_{\bf w}^n\backslash x_2)^t+I(H_{\bf w}\backslash x_2)+(x_2)$ and $(D,x_1)=I(P_{\bf w}^n\backslash x_1)^t+I(H_{\bf w})+(x_1)$
 by Lemma \ref{path1}, it follows that $\depth(S/(D:x_1))\geq \lceil\frac{(n-2)-t+1}{3}\rceil+1=\lceil\frac{n-t+2}{3}\rceil$  and $\depth(S/(D,x_1))\geq \lceil\frac{(n-1)-t+1}{3}\rceil=\lceil\frac{n-t}{3}\rceil$ by  Lemmas \ref{sum1}, \ref{path3} and  the statement (1). By  Lemma \ref{exact} and  the  exact sequence for $s=1$
\begin{gather}\label{eqn: equality12}
	\begin{matrix}
		0 & \longrightarrow & \frac{S}{D:x_s}(-1)  & \stackrel{ \cdot x_s} \longrightarrow  & \frac{S}{D} & \longrightarrow & \frac{S}{(D,x_s)} & \longrightarrow & 0,
	\end{matrix}
\end{gather}
we can get $\depth (S/D)\geq \lceil\frac{n-t}{3}\rceil$.

(ii) If $i\geq 3$, then $(D:x_2)=I(P_{\bf w}^n\backslash \{x_1,x_3\})^t+I(P_{\bf w}^i\backslash \{x_1,x_3\})+I(H_{\bf w}\backslash x_3)+(x_1,x_3)$ and $(D,x_2)=I(P_{\bf w}^n\backslash x_2)^t+I(P_{\bf w}^i\backslash x_2)+I(H_{\bf w})+(x_2)$ by Lemma \ref{path1}. By  induction,  Lemmas \ref{sum1},  \ref{path3} and the statement (1), we can obtain that
 $\depth(S/(D:x_2))\geq \lceil\frac{(n-3)-t}{3}\rceil+1=\lceil\frac{n-t}{3}\rceil$ and  $\depth(S/(D,x_2))\geq \lceil\frac{(n-2)-t}{3}\rceil+1=\lceil\frac{n-t+1}{3}\rceil$.
Using  Lemma \ref{exact} to the  short exact sequence (\ref{eqn: equality12}) for $s=2$,
we can get $\depth (S/D)\geq \lceil\frac{n-t}{3}\rceil$.
\end{proof}

	    \begin{Lemma}{\em (\cite[Lemma 3.10]{MTV})}\label{cycle3}
    		Let $C_{\bf w}^n$  be a trivially weighted $n$-cycle and  $H$  its non-empty subgraph. Then
		\[
\depth\left(S/(I(C_{\bf w}^n)^t+I(H))\right)\geq \lceil \frac{n-t+1}{3}\rceil \ \text{\ for all\ }t\ge 2.
\]
    \end{Lemma}

  \begin{Theorem}\label{cyclesubgraph}
   	Let  $n,t$ be two integers with $n\geq 4$ and  $1\leq t<\lceil \frac{n+1}{2}\rceil$. Let $C_{\bf w}^n$ be an $n$-cycle with only one edge of non-trivial weight and  $H_{\bf w}$ its subgraph with $V(H_{\bf w})=\{x_1,\ldots,x_{2t-2}\}$ and $E(H_{\bf w})=\bigcup\limits_{j=1}^{t-1}\{x_{2j-1}x_{2j}\}$. Assume that $H_{\bf w}^{'}$ is a non-empty subgraph of $H_{\bf w}$. Then
   	\[
   \depth(S/(I(C_{\bf w}^n)^t+I(H_{\bf w}^{'})))\geq \lceil \frac{n-t}{3}\rceil.
   \]
    \end{Theorem}
 \begin{proof}
	 We prove the statement by induction on $t$. Let $J=I(C_{\bf w}^n)^t+I(H_{\bf w}^{'})$. If $t=1$,
 then $J=I(C_{\bf w}^n)$.  By Theorem \ref{cycle},  $\depth(S/J)=\lceil \frac{n-1}{3}\rceil$. Therefore, in the following,  we can assume that $t\geq 2$. There are two cases:
	 	
 {\it Case $1$}:  If $E(H_{\bf w}^{'})=\{x_1x_2\}$, then $J=I(C_{\bf w}^n\backslash e_1)^t+(u_1)$. In this case, let $J_1=(J:x_2x_3)$ and $J_2=(J,x_2x_3)$.
	 First we  compute $\depth(S/J_1)$. Note that
	 	$(J_1,x_1)=I(C_{\bf w}^n\backslash x_1)^{t-1}+(x_1)$ by Lemma \ref{path1}, so $\depth(S/(J_1,x_1))=\lceil \frac{(n-1)-(t-1)+1}{3}\rceil=\lceil \frac{n-t+1}{3}\rceil$ by Lemmas \ref{sum1} and \ref{path}. Again using  Lemma \ref{path1}, we  can get
\[
(J_1:x_1)=\begin{cases}
	 		I(C_{\bf w}^n\backslash\{x_n,x_1\})+((x_1x_2)^{{\bf w}_1-1},x_n), &\text{if $t=2$,} \\ 	
	 		(I(C_{\bf w}^n\backslash e_1)^{t-1}:x_1)+((x_1x_2)^{{\bf w}_1-1}), &\text{if $t\geq 3$.}
	 	\end{cases}
\]
Consequently, if $t=2$, then $\depth(S/(J_1:x_1))\geq \lceil \frac{n-2}{3}\rceil$ by Lemmas \ref{sum1}, \ref{path} and \ref{nontrivialpath1}. Applying Lemma \ref{exact} to  the  short exact sequence
\begin{gather}\label{eqn: equality13}
	 		\begin{matrix}
	 			0 & \longrightarrow & \frac{S}{J_1:x_1}(-1)  & \stackrel{ \cdot x_1} \longrightarrow  & \frac{S}{J_1} & \longrightarrow & \frac{S}{(J_1,x_1)} & \longrightarrow & 0,
	 		\end{matrix}
	 	\end{gather}
we obtain  	$\depth (S/J_1)\geq \lceil \frac{n-2}{3}\rceil$.

If $t\geq 3$, then $(J_1:x_1x_n)=I(C_{\bf w}^n\backslash e_1)^{t-2}+((x_1x_2)^{{\bf w}_1-1})$, $((J_1:x_1),x_n)=I(C_{\bf w}^n\backslash \{x_1,x_n\})^{t-1}+((x_1x_2)^{{\bf w}_1-1},x_n)$ by Lemma \ref{path1}. By the induction and Lemmas \ref{sum1}, \ref{path3},  \ref{cycle3} and  Theorem \ref{pathsubgraph}, we get that $\depth(S/(J_1:x_1x_n))\geq \lceil \frac{n-(t-2)}{3}\rceil=\lceil \frac{n-t+2}{3}\rceil$ and
 $\depth(S/((J_1:x_1),x_n))\geq \lceil \frac{(n-1)-(t-1)+1}{3}\rceil=\lceil \frac{n-t+1}{3}\rceil$.
 It follows from the exact sequence (\ref{eqn: equality13}) and
\begin{gather*}
	 		\begin{matrix}
	 			0 & \longrightarrow & \frac{S}{J_1:x_1x_n}(-1)  & \stackrel{ \cdot x_n} \longrightarrow  & \frac{S}{J_1:x_1} & \longrightarrow & \frac{S}{((J_1:x_1),x_n)} & \longrightarrow & 0
	 		\end{matrix}
	 	\end{gather*}
that  $\depth (S/(J_1:x_1))\geq \lceil \frac{n-t}{3}\rceil$ and $\depth (S/J_1)\geq \lceil \frac{n-t}{3}\rceil$.

	 	Next, we compute $\depth(S/J_2)$. Note that  $(J_2:x_3x_4)=I(C_{\bf w}^n\backslash x_2)^{t-1}+(x_2)$,  $(J_2,x_3)=I(C_{\bf w}^n\backslash \{x_2,x_3\})^t+(u_1,x_3)$ and $((J_2:x_3),x_4)=I(C_{\bf w}^n\backslash \{x_2,x_4\})^t+(x_2,x_4)$ by Lemma \ref{path1}, we have that
 $\depth(S/(J_2,x_3))\geq \lceil \frac{n-t+1}{3}\rceil$, $\depth(S/(J_2:x_3x_4))= \lceil \frac{(n-1)-(t-1)+1}{3}\rceil=\lceil \frac{n-t+1}{3}\rceil$ and $\depth(S/((J_2:x_3),x_4))=1+\lceil \frac{(n-3)-t+1}{3}\rceil=\lceil \frac{n-t+1}{3}\rceil$ by  Lemmas \ref{sum1}, \ref{path} and Theorem \ref{pathsubgraph}. Using the   exact sequences
	 	\begin{gather*}
	 		\begin{matrix}
	 			0 & \longrightarrow & \frac{S}{J_2:x_3}(-1)  & \stackrel{ \cdot x_3} \longrightarrow  & \frac{S}{J_2} & \longrightarrow & \frac{S}{(J_2,x_3)} & \longrightarrow & 0,\\
	 			0 & \longrightarrow & \frac{S}{J_2:x_3x_4}(-1)  & \stackrel{ \cdot x_3} \longrightarrow  & \frac{S}{J_2:x_3} & \longrightarrow & \frac{S}{((J_2:x_3),x_4)} & \longrightarrow & 0,
	 		\end{matrix}
	 	\end{gather*}
	 	we can get $\depth(S/J_2)\geq \lceil \frac{n-t}{3}\rceil$.
	 	
	 	As a result,  using the short exact sequence
	 	\begin{gather*}
	 		\begin{matrix}
	 			0 & \longrightarrow & \frac{S}{J_1}(-2)  & \stackrel{ \cdot x_2x_3} \longrightarrow  & \frac{S}{J} & \longrightarrow & \frac{S}{J_2} & \longrightarrow & 0,
	 		\end{matrix}
	 	\end{gather*}
	 	we can obtain  $\depth(S/J)\geq \lceil \frac{n-t}{3}\rceil$.
	 	
	 {\it Case $2$}:   If $E(H_{\bf w}^{'})\neq \{x_1x_2\}$, then we let $m=\max\{ i\in [t-2]\mid x_{2i+1}x_{2i+2}\in E(H_{\bf w}^{\prime})\}$ and  $K_j=J+\sum\limits_{s=2m+2}^{j}(x_sx_{s+1})$ for $2m+2\leq j\leq n$.
	 	Thus $(J:x_{2m+2}x_{2m+3})=I(C_{\bf w}^n\backslash x_{2m+1})^{t-1}+I(H_{\bf w}^{\prime}\backslash x_{2m+1})+(x_{2m+1})$ and $(K_{j}:x_{j+1}x_{j+2})=I(C_{\bf w}^n\backslash x_j)^{t-1}+I(H_{\bf w}^{\prime}\backslash x_j)+(x_j)$ for  $2m+2\leq j\leq n-1$. By Theorem \ref{pathsubgraph}, we get that $\depth(S/(J:x_{2m+2}x_{2m+3}))\geq \lceil \frac{(n-1)-(t-1)}{3}\rceil=\lceil \frac{n-t}{3}\rceil$, $\depth(S/(K_{j}:x_{j+1}x_{j+2}))\geq \lceil \frac{(n-1)-(t-1)}{3}\rceil=\lceil \frac{n-t}{3}\rceil$ for  $2m+2\leq j\leq n-1$. It follows from the exact sequences
	 	\begin{gather*}
	 		\begin{matrix}
	 			0 & \longrightarrow & \frac{S}{J:x_{2m+2}x_{2m+3}}(-2)  & \xrightarrow[]{ \cdot x_{2m+2}x_{2m+3}}  & \frac{S}{J} & \longrightarrow & \frac{S}{K_{2m+2}} & \longrightarrow & 0,\\
	 			0 & \longrightarrow & \frac{S}{K_{2m+2}:x_{2m+3}x_{2m+4}}(-2)  & \xrightarrow[]{\cdot x_{2m+3}x_{2m+4}} & \frac{S}{K_{2m+2}} & \longrightarrow & \frac{S}{K_{2m+3}} & \longrightarrow & 0,\\
	 			0 & \longrightarrow & \frac{S}{K_{2m+3}:x_{2m+4}x_{2m+5}}(-2)  & \xrightarrow[]{ \cdot x_{2m+4}x_{2m+5}} & \frac{S}{K_{2m+3}} & \longrightarrow & \frac{S}{K_{2m+4}} & \longrightarrow & 0,\\
	 			&  &\vdots&  &\vdots&  &\vdots&  &\\
	 			0 & \longrightarrow & \frac{S}{K_{n-1}:x_nx_1}(-2)  & \xrightarrow[]{ \cdot x_{n}x_{1}} & \frac{S}{K_{n-1}} & \longrightarrow & \frac{S}{K_{n}} & \longrightarrow & 0,\\
	 		\end{matrix}
	 	\end{gather*}
	 that $\depth(S/J)\geq \lceil \frac{n-t}{3}\rceil$ if $\depth(S/K_n)\geq \lceil \frac{n-t}{3}\rceil$. It remains to prove that
$\depth(S/K_n)\geq \lceil \frac{n-t}{3}\rceil$.

Let $H_{\bf w}^1$ be a subgraph of $C_{\bf w}^n$ with $V(H_{\bf w}^{1})=V(H_{\bf w}^{'})\cup\{x_{2m+3},\ldots,x_n,x_1\}$, $E(H_{\bf w}^{1})=E(H_{\bf w}^{'})\cup(\bigcup\limits_{i=2m+2}^{n}\{x_ix_{i+1}\})$, $p_1=\max\{ j\in [n]\mid x_jx_{j+1}\notin E(H_{\bf w}^{1})\}$ and $\frakP_1=I(C_{\bf w}^n)^t+I(H_{\bf w}^{1})$.
For $i\ge 2$, let $H_{\bf w}^i$ be a subgraph of $C_{\bf w}^n$ with $V(H_{\bf w}^{i})=V(H_{\bf w}^{i-1})\cup\{x_{p_{i-1}},x_{p_{i-1}+1}\}$, $E(H_{\bf w}^{i})=E(H_{\bf w}^{i-1})\cup \{x_{p_{i-1}}x_{p_{i-1}+1}\}$, $p_i=\max\{j\in [n]\mid x_jx_{j+1}\notin E(H_{\bf w}^{i})\}$ and $\frakP_i=I(C_{\bf w}^n)^t+I(H_{\bf w}^{i})$. Since $x_2x_3\notin E(H_{\bf w}^{'})$, $p_i\ge 2$. By the definition of $p_i$ in above iterative process of $H_{\bf w}^{i}$, we know that $p_\ell=2$ for some  $\ell\geq 2$.

By Lemma \ref{path1},  we obtain that   for any $i\in [\ell-1]$,
\[
(\frakP_i:x_ix_{i+1})=\begin{cases}
	I(C_{\bf w}^n\backslash x_{i+2})^{t-1}+I(G_{\bf w}\backslash x_{i+2})+(x_{i+2}), &\text{if $x_{i-1}x_{i}\notin E(H_{\bf w}^{i})$,} \\ 	
	E, &\text{if $x_{i-1}x_{i}\in E(H_{\bf w}^{i})$,}
\end{cases}
\]
where $E=I(C_{\bf w}^n\backslash \{x_{i-1},x_{i+2}\})^{t-1}+I(G_{\bf w}\backslash \{x_{i-1},x_{i+2}\})+(x_{i-1},x_{i+2})$.

Similarly  calculation of  $\depth(S/(((A:x_{n-1}),x_n):x_{n-2}))$  as in   {\it Case $1$} of the proof of Theorem \ref{pathsubgraph}(1), we deduce
$\depth(S/(\frakP_i:x_ix_{i+1}))\geq \lceil \frac{n-t}{3}\rceil$ for any $i\in [\ell-1]$.
To prove that  $\depth(S/K_n)\geq \lceil \frac{n-t}{3}\rceil$, by the following sequences
\begin{gather}\label{eqn: equality14}
	\begin{matrix}
		0 & \longrightarrow & \frac{S}{\frakP_1:x_{p_1}x_{p_1+1}}(-2)  & \xrightarrow[]{ \cdot x_{p_1}x_{p_1+1}}  & \frac{S}{K_n} & \longrightarrow & \frac{S}{\frakP_2} & \longrightarrow & 0,\\
		0 & \longrightarrow & \frac{S}{\frakP_2:x_{p_2}x_{p_2+1}}(-2)  & \xrightarrow[]{ \cdot x_{p_2}x_{p_2+1}}   & \frac{S}{\frakP_2} & \longrightarrow & \frac{S}{\frakP_3} & \longrightarrow & 0,\\
		&  &\vdots&  &\vdots&  &\vdots&  &\\
		0 & \longrightarrow & \frac{S}{\frakP_{\ell-1}:x_{p_{\ell-1}}x_{p_{\ell-1}+1}}(-2)  & \xrightarrow[]{ \cdot x_{p_{\ell-1}}x_{p_{\ell-1}+1}}   & \frac{S}{\frakP_{\ell-1}} & \longrightarrow & \frac{S}{\frakP_{\ell}} & \longrightarrow & 0,
	\end{matrix}
\end{gather}
we need to compute  $\depth(S/\frakP_\ell)$. We consider the following two cases:

(a) If ${x_1x_2}\notin E(H_{\bf w}^{'})$, then $(\frakP_\ell:x_2x_3)=I(C_{\bf w}^n\backslash x_4)^{t-1}+I(H^{\ell}_{\bf w}\backslash x_4)+(x_4)$ and $(\frakP_\ell,x_2x_3)=I(C_{\bf w}^n\backslash e_1)+((x_1x_2)^{t{\bf w}_1})$ by Lemma \ref{path1}. Thus $\depth(S/(\frakP_\ell:x_2x_3))\geq \lceil \frac{(n-1)-(t-1)}{3}\rceil=\lceil \frac{n-t}{3}\rceil$, $\depth(S/(\frakP_\ell,x_2x_3))=\lceil \frac{n-1}{3}\rceil$  by Lemma \ref{cycle} and Theorem \ref{pathsubgraph}.  We can get $\depth(S/\frakP_\ell)\geq \lceil\frac{n-t}{3}\rceil$ by   the exact sequence
\[
0  \longrightarrow \frac{S}{\frakP_\ell:x_2x_3}(-2)   \stackrel{ \cdot x_2x_3} \longrightarrow   \frac{S}{\frakP_\ell} \longrightarrow  \frac{S}{(\frakP_\ell,x_2x_3)}  \longrightarrow  0.
\]

(b) If ${x_1x_2}\in E(H_{\bf w}^{'})$, then $\frakP_\ell=I(C_{\bf w}^n\backslash e_2)+((x_2x_3)^t)$. It follows that  $(\frakP_\ell,x_2x_3)=I(C_{\bf w}^n)$, $(\frakP_\ell:x_2x_3x_1)=I(C_{\bf w}^n\backslash \{x_2,x_4,x_n\})+((x_1x_2)^{{\bf w}_1-1},(x_2x_3)^{t-1})+(x_4,x_n)$, $((\frakP_\ell:x_2x_3),x_1)=I(C_{\bf w}^n\backslash \{x_1,x_2,x_4\})+((x_2x_3)^{t-1})+(x_1,x_4)$ by Lemma \ref{path1}. Thus  $\depth(S/(\frakP_\ell,x_2x_3))=\lceil \frac{n-1}{3}\rceil$, $\depth(S/(\frakP_\ell:x_2x_3x_1))=\lceil \frac{n-5}{3}\rceil+1=\lceil \frac{n-2}{3}\rceil$, $\depth(S/((\frakP_\ell:x_2x_3),x_1))=\lceil \frac{n-4}{3}\rceil+1=\lceil \frac{n-1}{3}\rceil$ by Lemmas \ref{sum1}, \ref{path}, \ref{nontrivialpath1} and Theorem \ref{cycle}. By  Lemma \ref{exact} and  the following short exact sequences
\begin{gather*}
	\begin{matrix}
		0 & \longrightarrow & \frac{S}{\frakP_\ell:x_2x_3}(-2)  & \stackrel{ \cdot x_2x_3} \longrightarrow  & \frac{S}{\frakP_\ell} & \longrightarrow & \frac{S}{(\frakP_\ell,x_2x_3)} & \longrightarrow & 0,\\
		0 & \longrightarrow & \frac{S}{\frakP_\ell:x_2x_3x_1}(-1)  & \stackrel{ \cdot x_1} \longrightarrow  & \frac{S}{\frakP_\ell:x_2x_3} & \longrightarrow & \frac{S}{((\frakP_\ell:x_2x_3),x_1)} & \longrightarrow & 0,
	\end{matrix}
\end{gather*}
we can get $\depth(S/\frakP_\ell)\geq \lceil\frac{n-t}{3}\rceil$.
Again using the exact  sequences  (\ref{eqn: equality14}), we obtain	$\depth(S/K_n)\geq \lceil\frac{n-t}{3}\rceil$
   and finish the proof.
\end{proof}

  	 \begin{Theorem}\label{1n-cyclef1}
	 	Let  $C_{\bf w}^n$ be an $n$-cycle with $n\geq 4$, and ${\bf w}_1\geq 2$ and ${\bf w}_i=1$ for any $i\neq 1$. Let $2\leq t < \lceil\frac{n+1}{2}\rceil$ be an integer  and $f=(x_1x_2)^{{\bf w}_1}\prod\limits_{k=3}^{2t-2}x_k$, then
\[
\depth(S/(I(C_{\bf w}^n)^t,f))\geq \lceil\frac{n-t}{3}\rceil.
\]
	 \end{Theorem}
\begin{proof}
	 	Let $f_j=(x_1x_2)^{{\bf w}_1}\prod\limits_{k=3}^{2t-2j}x_k$ for each $j\in [t-2]$, and let $H_{\bf w}$ be any  subgraph of $C_{\bf w}^n$ whose edges are taken from the set
 $\{e_{2t-2j+1},e_{2t-2j+3},\ldots,e_{2t-1}\}$. then $f_1=f$,  $f_j=x_{2t-2j-1}x_{2t-2j}f_{j+1}$ and
$I(C_{\bf w}^n)^t+I(H_{\bf w})+(f_j)=(I(C_{\bf w}^n)^t+I(H_{\bf w})+(f_{j+1}))\cap (I(C_{\bf w}^n)^t+I(H_{\bf w})+(x_{2t-2j-1}x_{2t-2j}))$.
The desired results follow from exact sequence
\[
0  \longrightarrow  \frac{S}{(A,f_j)}  \longrightarrow   \frac{S}{(A,f_{j+1})}\oplus\frac{S}{(A,x_{2t-1-2j}x_{2t-2j})} \longrightarrow \frac{S}{(A,f_{j+1},x_{2t-1-2j}x_{2t-2j})}  \longrightarrow  0,
\]
where $A=I(C_{\bf w}^n)^t+I(H_{\bf w})$ for  any $1\leq j\leq t-2$.
 \end{proof}	

   \begin{Theorem}\label{1n-cyclef2}
     	Let  $t$, $f$ and $C_{\bf w}^n$ be  as in Theorem \ref{1n-cyclef1}.  Then
     	\[
     (I(C_{\bf w}^n)^t:f)=I(C_{\bf w}^n)+(x_n^2)+(\bigcup\limits_{i=0}^{t-2}\{x_3x_{2i+3}\})+(\bigcup\limits_{j=1}^{t-1}\{x_{2j+1}x_{n}\})+(\bigcup\limits_{p=1}^{t-2}\bigcup\limits_{q=0}^{t-p-2}\{x_{2p}x_{2p+2q+3}\}).
     \]
     \end{Theorem}
 \begin{proof}
     	Let $I=I(C_{\bf w}^n)$.  First we prove that the inclusion relation ``$\supseteq$" holds. It is easy to see that $(I^t:f)\supseteq I$ and $(x_n^2)\in(I^t:f)$, since $f=(x_1x_2)^{{\bf w}_1}\prod\limits_{k=3}^{2t-2}x_k=u_1\prod\limits_{\ell=1}^{t-2}u_{2\ell+1}\in I^{t-1}$,
  and $x_nf=x_1^{{\bf w}_1-2}x_2^{{\bf w}_1}u_n^2(\prod\limits_{\ell=1}^{t-2}u_{2\ell+1})
     \in I^t$.
     Now we divide into three cases:
 (i) If $p\in [t-2]$ and $q\in [t-p-2]\cup \{0\}$, then
$ x_{2p}x_{2p+2q+3}f= u_1(\prod\limits_{\ell=1}^{p-1}u_{2\ell+1})(\prod\limits_{\ell=p}^{p+q+1}u_{2\ell})(\prod\limits_{\ell=p+q+1}^{t-2}u_{2\ell+1})\in I^t$;
(ii) If $i\in [t-2]\cup \{0\}$, then $x_3x_{2i+3}f=x_1^{{\bf w}_1}x_2^{{\bf w}_1-2}u_2^2(\prod\limits_{\ell=2}^{i+1}u_{2\ell})(\prod\limits_{\ell=i+1}^{t-2}u_{2\ell+1})\in I^t$;
(iii) If $j\in [t-1]$, then $x_{2j+1}x_nf=(x_1x_2)^{{\bf w}_1-1}u_2\cdot u_n(\prod\limits_{\ell=2}^{j}u_{2\ell})(\prod\limits_{\ell=j}^{t-2}u_{2\ell+1})\in I^t$.
Therefore, the  inclusion relation ``$\supseteq$" holds.

 Next, we prove  the inclusion relation ``$\subseteq$".
Let $u \in \mathcal{G}(I^n: f)$ be a monomial, then $uf=h\prod\limits_{i=1}^{n}u_i^{k_i}$ for some monomial $h$, where $\sum\limits_{i=1}^{n}k_i=t$ and each $k_i\geq 0$.
Thus
     	\begin{align}\label{eqn: equality15}
     		u(x_1x_2)^{{\bf w}_1}\prod\limits_{k=3}^{2t-2}x_k=h\prod\limits_{i=1}^{n}u_i^{k_i}.
     	\end{align}
   Therefore, if $k_i\geq 1$ for some $2t-1\leq i\leq n-1$, then $u\in(x_ix_{i+1})\subseteq I(C_{\bf w}^n)$, and if $k_2\geq 3$, then $u\in (x_3^2)$   by comparing
the degree of $x_3$ in Eq. (\ref{eqn: equality15}).  If $k_i\geq 2$ for some $i\in \{3,4,\ldots,2t-2\}\cup \{1,n\}$, then $u\in (x_n^2)+I(C_{\bf w}^n)$. It remain to consider the cases:
 $k_i=0$ for all $2t-1\leq i\leq n-1$ and  $k_2\leq 2$ and  $k_i\leq 1$ for all $i\in \{3,4,\ldots,2t-2\}\cup \{1,n\}$.  In this case, Eq. (\ref{eqn: equality15}) becomes
     	\begin{align}\label{eqn: equality16}
     		u(x_1x_2)^{{\bf w}_1}\prod\limits_{k=3}^{2t-2}x_k=hu_n^{k_n}\prod\limits_{i=1}^{2t-2}u_i^{k_i},
     	\end{align}
     	where  $k_n+\sum\limits_{i=1}^{2t-2}k_i=t$. We distinguish into two subcases:
     	
     	(a) When $k_2=2$,  Eq. (\ref{eqn: equality16}) becomes
     $u(x_1x_2)^{{\bf w}_1}\prod\limits_{k=3}^{2t-2}x_k=hu_2^2u_1^{k_1}u_n^{k_n}\prod\limits_{i=3}^{2t-2}u_i^{k_i}$. Furthermore,
     if $k_1=1$, then  $u\prod\limits_{k=3}^{2t-2}x_k=hu_2^2u_n^{k_n}\prod\limits_{i=3}^{2t-2}u_i^{k_i}$, which forces $u\in (x_2x_3)$. It is easy to see that $u\in (x_3x_{2t-1})$ if $k_{2t-2}=1$, and
     $u\in (x_3x_n)$ if $k_n=1$.
   If $k_1=k_{2t-2}=k_n=0$, then Eq. (\ref{eqn: equality16}) becomes
     $u(x_1x_2)^{{\bf w}_1}\prod\limits_{k=3}^{2t-2}x_k=h\prod\limits_{i=2}^{2t-3}u_i^{k_i}$.  In this case,   $k_3=1$ or $k_{2i}=k_{2i+1}=1$ for some $i\in [t-2]\backslash\{1\}$. In fact, if
      $k_3=0$, and
      $k_{2i}=0$ or $k_{2i+1}=0$ for any $i\in [t-2]\backslash\{1\}$, then  $\sum\limits_{i=2}^{2t-3}k_i\le (2t-3)-(t-2)= t-1$, a contradiction.  By comparing the degrees of $x_3$ and $x_{2i+1}$ in the above expression, we get that $u=\frac{h\prod\limits_{i=2}^{2t-3}u_i^{k_i}}{(x_1x_2)^{{\bf w}_1}\prod\limits_{k=3}^{2t-2}x_k}\in (\bigcup\limits_{i=0}^{t-3}\{x_3x_{2i+3}\})$.

     	(b) If $0\leq k_2\leq 1$, then there are three subcases:
     	
     	(i) When $k_n=1$,   Eq. (\ref{eqn: equality16}) becomes $u(x_1x_2)^{{\bf w}_1}\prod\limits_{k=3}^{2t-2}x_k=hu_n\prod\limits_{i=1}^{2t-2}u_i^{k_i}$. Furthermore,  if $k_1=1$ then $u\in(x_nx_1)\subseteq I(C_{\bf w}^n)$, and if $k_{2t-2}=1$ then $u\in (x_{2t-1}x_n)\subseteq (\bigcup\limits_{j=1}^{t-1}\{x_{2j+1}x_{n}\})$.
   If $k_1=k_{2t-2}=0$, then Eq. (\ref{eqn: equality16}) becomes
     	$u(x_1x_2)^{{\bf w}_1}\prod\limits_{k=3}^{2t-2}x_k=hu_n\prod\limits_{i=2}^{2t-3}u_i^{k_i}$.
     	In this case,  $k_{2i}=k_{2i+1}=1$ for some $i\in [t-2]$. In fact, if $k_{2i}=0$ or $k_{2i+1}=0$ for any $i\in [t-2]$, then  $\sum\limits_{i=2}^{2t-3}k_i\leq (2t-4)-(t-2)=t-2$, a contradiction. By comparing the degrees of $x_n$ and $x_{2i+1}$ in the above expression, we get that $u=\frac{hu_n\prod\limits_{i=2}^{2t-3}u_i^{k_i}}{(x_1x_2)^{{\bf w}_1}\prod\limits_{k=3}^{2t-2}x_k}\in (\bigcup\limits_{i=1}^{t-2}\{x_{2i+1}x_{n}\})$.

     	(ii) When $k_n=0$ and $k_{2t-2}=1$,   Eq. (\ref{eqn: equality16}) becomes $u(x_1x_2)^{{\bf w}_1}\prod\limits_{k=3}^{2t-2}x_k=hu_{2t-2}\prod\limits_{i=1}^{2t-3}u_i^{k_i}$. By comparing the
degrees of $x_{2t-2}$  and $x_{2t-1}$, we can get that  $u\in(x_{2t-2}x_{2t-1})\subseteq I(C_{\bf w}^n)$ if $k_{2t-3}=1$.
   If $k_{2t-3}=0$, then Eq. (\ref{eqn: equality16}) becomes
   $u(x_1x_2)^{{\bf w}_1}\prod\limits_{k=3}^{2t-2}x_k=hu_{2t-2}\prod\limits_{i=1}^{2t-4}u_i^{k_i}$.  By arguments similar to the case (a) above, we deduce
$k_{2i-1}=k_{2i}=1$  for some $i\in [t-2]$. Again comparing the
degrees of $x_{2i}$ and $x_{2t-1}$ in the above expression, we  can get that $u\in \bigcup\limits_{i=1}^{t-2}(x_{2i}x_{2t-1})\subseteq (\bigcup\limits_{p=1}^{t-2}\bigcup\limits_{q=0}^{t-p-2}\{x_{2p}x_{2p+2q+3}\})$.

     	(iii) When $k_n=k_{2t-2}=0$,  Eq. (\ref{eqn: equality16}) becomes
     	\begin{align}\label{eqn: equality17}
     		u(x_1x_2)^{{\bf w}_1}\prod\limits_{k=3}^{2t-2}x_k=h\prod\limits_{i=1}^{2t-3}u_i^{k_i}.
     	\end{align}
     	Likewise,   $k_{2i-1}=k_{2i}=1$ for some $i\in [t-2]$.  Let $a=\min\{i\in [t-2]\mid k_{2i-1}=k_{2i}=1\}$. We divide into two subcases:
     	
     	(a)	If $a=1$, then  $k_{2b}=k_{2b+1}=1$ for some $b\in [t-2]$. Otherwise,  $\sum\limits_{j=1}^{2t-3}k_i\leq  (2t-3)-(t-2)=t-1$, a contradiction.
     By comparing the degrees of $x_2$ and $x_{2b+1}$ in Eq. (\ref{eqn: equality17}), we can get 	$u\in (x_2x_{2b+1})\subseteq I(C_{\bf w}^n)+ (\bigcup\limits_{p=1}^{t-2}\bigcup\limits_{q=0}^{t-p-2}\{x_{2p}x_{2p+2q+3}\})$.
       	
     	(b) If $2\leq a\leq t-2$, then  $k_{2i-1}=0$ or $k_{2i}=0$ for any $i\in [a-1]$. So
      $k_{2b}=k_{2b+1}=1$ for some $b\in [t-2]\backslash[a-1]$. Otherwise,  $\sum\limits_{j=1}^{2t-3}k_i\leq(2t-3)-(a-1)-(t-a-1)=t-1$, a contradiction.
     		By comparing the degrees of $x_{2a}$ and $x_{2b+1}$ in Eq. (\ref{eqn: equality17}), we can get that $u\in (x_{2a}x_{2b+1})\subseteq I(C_{\bf w}^n)+ (\bigcup\limits_{p=1}^{t-2}\bigcup\limits_{q=0}^{t-p-2}\{x_{2p}x_{2p+2q+3}\})$.
     \end{proof}

\begin{Theorem}\label{1n-cycle}
	Let  $t$  and $C_{\bf w}^n$ be  as in Theorem \ref{1n-cyclef1}, and let
	\[
	f=
	\begin{cases}
		x_3(x_1x_2)^{{\bf w}_1}, &\text{if $t=2$,} \\ 	
		x_3(x_1x_2)^{{\bf w}_1}\prod\limits_{k=3}^{t}(x_kx_{k+1}), &\text{if $3 \leq t<\lceil\frac{n+1}{2}\rceil$,} \\
	\end{cases}
	\]
	then   $(I(C_{\bf w}^n)^t:f)=I(C_{\bf w}^n)+\sum\limits_{k=2}^{t+2}(x_k)+(x_n)$.
\end{Theorem}

\begin{proof}
	Let $I=I(C_{\bf w}^n)$.
 If $t=2$, then it is trival that $I\subseteq (I^2:f)$. Since $x_2f=(x_1x_2)^{{\bf w}_1}(x_2x_3)\in I^2$, $x_3f=x_1^{{\bf w}_1}x_2^{{{\bf w}_1}-2}(x_2x_3)^2\in I^2$, $x_4f=(x_1x_2)^{{\bf w}_1}(x_3x_4)\in I^2$ and $x_nf=(x_1x_2)^{{\bf w}_1-1}(x_2x_3)(x_nx_1)\in I^2$, it follows that $\sum\limits_{k=2}^{4}(x_k)+(x_n)\subseteq(I^2:f)$.
	
	On the other hand, let  $u \in \mathcal{G}(I^2: f)$ be a monomial, then $uf=h\prod\limits_{i=1}^{n}u_i^{k_i}$ for some monomial $h$, where $\sum\limits_{i=1}^{n}k_i=2$ and each $k_i\geq 0$. Thus
	\begin{align}\label{43}
		ux_3(x_1x_2)^{{\bf w}_1}=h\prod\limits_{i=1}^{n}u_i^{k_i}.
	\end{align}
	If  $k_i\geq 1$ for some $i\ge 3$, then $u\in I(C_{\bf w}^n)+(x_4,x_n)$. Next, we assume that $k_i=0$ for any $i\ge 3$. Then  Eq. (\ref{43}) becomes $ux_3(x_1x_2)^{{\bf w}_1}=h\prod\limits_{i=1}^{2}u_i^{k_i}$. It is easy to see that $u\in (x_2)$ if $k_1=2$ and $k_2=0$, or  $k_1=1$ and $k_2=1$, and that $u\in (x_3)$ if $k_1=0$ and $k_2=2$.
	
If $3 \leq t<\lceil\frac{n+1}{2}\rceil$, then $(I^t:f)=I(C_{\bf w}^n)+\sum\limits_{k=2}^{t+2}(x_k)+(x_n)$ as  similarly discussed in the proof of  Theorem \ref{2n-cycle}(1).
\end{proof}

\begin{Theorem}\label{}
 Let $1 \leq t<\lceil\frac{n+1}{2}\rceil$  be an integer and    $C_{\bf w}^n$    as in Theorem \ref{1n-cyclef1}, then $\depth (S/I(C_{\bf w}^n)^t)=\lceil\frac{n-t}{3}\rceil$.
\end{Theorem}

\begin{proof}
Let $I=I(C_{\bf w}^n)$.  If $t=1$, then the desired result  follows from Theorem \ref{cycle}. Now we assume that $2\leq t<\lceil\frac{n+1}{2}\rceil$. If $t=2$, we  choose $f=x_3(x_1x_2)^{{\bf w}_1}$, otherwise, we  choose $f=x_3(x_1x_2)^{{\bf w}_1}\prod\limits_{k=3}^{t}(x_kx_{k+1})$. Thus, by Lemma \ref{up} and Theorem \ref{1n-cycle}, $\depth (S/I^t)\leq\depth (S/(I^t:f))=1+\lceil\frac{(n-1)-(t+2)}{3}\rceil=\lceil\frac{n-t}{3}\rceil$.

 It remains to  prove  $\depth (S/I^t)\geq \lceil\frac{n-t}{3}\rceil$. Let $g=(x_1x_2)^{{\bf w}_1}\prod\limits_{k=3}^{2t-2}x_k$ and  $\frakQ=(I^t:g)$.
 By the sequence (\ref{eqn: equality1}) in Lemma \ref{up} and Theorem \ref{1n-cyclef1}, it suffices to prove that  $\depth(S/\frakQ)\geq \lceil\frac{n-t}{3}\rceil$.
By Lemma \ref{1n-cyclef2}, we deduce
	\begin{align*}
		(\frakQ:x_3)&=[I+(x_n^2)+(\bigcup\limits_{i=0}^{t-2}\{x_3x_{2i+3}\})+(\bigcup\limits_{j=1}^{t-1}\{x_{2j+1}x_{n}\})+(\bigcup\limits_{p=1}^{t-2}\bigcup\limits_{q=0}^{t-p-2}\{x_{2p}x_{2p+2q+3}\})]:x_3\\
		&=I(C_{\bf w}^n\backslash\{x_2,x_4\})+(\bigcup\limits_{i=0}^{t-2}\{x_{2i+3}\})+(\bigcup\limits_{p=1}^{t-2}\bigcup\limits_{q=0}^{t-p-2}\{x_{2p}x_{2p+2q+3}\})+(x_2,x_4,x_n)\\
		&=I(P_{\bf w}^{n-2t})+(\bigcup\limits_{i=0}^{t-2}\{x_{2i+3}\})+(x_2,x_4,x_n)
	\end{align*}
	where $P_{\bf w}^{n-2t}$ is an induced subgraph of $C_{\bf w}^n$ on the set $\{x_{2t},\ldots,x_{n-1}\}$, and
 \begin{align*}
 	((\frakQ,x_3):x_n)&=[I+(x_3,x_n^2)+(\bigcup\limits_{j=2}^{t-1}\{x_{2j+1}x_{n}\})+(\bigcup\limits_{p=1}^{t-2}\bigcup\limits_{q=0}^{t-p-2}\{x_{2p}x_{2p+2q+3}\})]:x_n\\
 	&=I(C_{\bf w}^n\backslash\{x_1,x_{n-1}\})+(\bigcup\limits_{j=2}^{t-1}\{x_{2j+1}\})+(\bigcup\limits_{p=1}^{t-2}\bigcup\limits_{q=0}^{t-p-2}\{x_{2p}x_{2p+2q+3}\})\\
 	&+(x_1,x_3,x_{n-1},x_n)\\
 	&=I(P_{\bf w}^{n-2t-1})+(\bigcup\limits_{j=1}^{t-1}\{x_{2j+1}\})+(x_1,x_{n-1},x_n),
 \end{align*}
 where $P_{\bf w}^{n-2t-1}$ is an induced subgraph of $C_{\bf w}^n$ on the set $\{x_{2t},\ldots,x_{n-2}\}$.
 Therefore, by Lemmas \ref{sum1}, \ref{path} and \ref{nontrivialpath1}, we obtain that  $\depth(S/(\frakQ:x_3))\geq \lceil\frac{n-t}{3}\rceil$ and $\depth(S/((\frakQ,x_3):x_n))\geq \lceil\frac{n-t}{3}\rceil$. We consider the following two cases:

 (1) If $t=2$, then  $(\frakQ,x_3,x_n)=I(C_{\bf w}^n)+(x_3,x_n)$ by Lemma \ref{1n-cyclef2}. Thus, we have  $\depth(S/(\frakQ,x_3,x_n))=\lceil\frac{n-4}{3}\rceil+1=\lceil\frac{n-1}{3}\rceil$. Using Lemma \ref{exact} to  the following  exact sequences
\begin{gather}\label{eqn: equality19}
		\begin{matrix}
			0 & \longrightarrow & \frac{S}{\frakQ:x_3}(-1) & \stackrel{\cdot x_3}\longrightarrow & \frac{S}{\frakQ} & \longrightarrow & \frac{S}{(\frakQ,x_3)} & \longrightarrow & 0,\\
			0 & \longrightarrow & \frac{S}{(\frakQ,x_3):x_n}(-1) & \stackrel{\cdot x_n} \longrightarrow & \frac{S}{(\frakQ,x_3)} & \longrightarrow & \frac{S}{(\frakQ,x_3,x_n)} & \longrightarrow & 0,
		\end{matrix}
	\end{gather}
  we can get $\depth(S/\frakQ)\geq \lceil\frac{n-t}{3}\rceil$.
	
	(2) If $3 \leq t<\lceil\frac{n+1}{2}\rceil$, then let $M_{j}=(((\frakQ,x_3,x_n)+(\bigcup\limits_{s=j+1}^{t}\{x_{2s-1}\})):x_{2j-1})$ and $H_{j}=(\frakQ,x_3,x_n)+(\bigcup\limits_{s=j}^{t}\{x_{2s-1}\})$ for any $3\leq j\leq t$, where $(\bigcup\limits_{s=j}^{t}\{x_{2s-1}\})=0$ if $j=t$.  By Lemma \ref{1n-cyclef2}, we get that $H_3=I(P_{\bf w}^{n-2t})+(\bigcup\limits_{j=3}^{t}\{x_{2j-1}\})+(u_1,x_3,x_n)$ and $M_j=I(P_{\bf w}^{n-2t})+(\bigcup\limits_{s=j+1}^{t}\{x_{2s-1}\})+(\bigcup\limits_{m=1}^{j-2}\{x_{2m}\})+(x_3,x_{2j-2},x_{2j},x_n)$ for any $3\leq j\leq t$,
	where $P_{\bf w}^{n-2t}$ is an induced subgraph of $C_{\bf w}^n$ on $\{x_{2t},\ldots,x_{n-1}\}$.
 Therefore, we can deduce that $\depth(S/H_3)\geq \lceil\frac{n-t}{3}\rceil$ and  $\depth(S/M_j)\geq \lceil\frac{n-t}{3}\rceil$.
 By the following  exact sequences
 \begin{gather*}
		\begin{matrix}
			0 & \longrightarrow & \frac{S}{M_t}(-1) & \stackrel{\cdot x_{2t-1}} \longrightarrow & \frac{S}{(\frakQ,x_3,x_n)} & \longrightarrow & \frac{S}{H_t} & \longrightarrow & 0,\\
			0 & \longrightarrow & \frac{S}{M_{t-1}}(-1) & \stackrel{\cdot x_{2t-3}} \longrightarrow & \frac{S}{H_t} & \longrightarrow & \frac{S}{H_{t-1}} & \longrightarrow & 0,\\
			&  &\vdots&  &\vdots&  &\vdots&  &\\
			0 & \longrightarrow & \frac{S}{M_3}(-1) & \stackrel{\cdot x_5}\longrightarrow & \frac{S}{M_4} & \longrightarrow & \frac{S}{H_3} & \longrightarrow & 0,
		\end{matrix}
	\end{gather*}
we obtain  $\depth(S/(\frakQ,x_3,x_n))\geq \lceil\frac{n-t}{3}\rceil$. Again using the  exact sequence (\ref{eqn: equality19}), we get that $\depth (S/\frakQ)\geq \lceil\frac{n-t}{3}\rceil$.
\end{proof}

\begin{Theorem}\label{odd-cycle}
Let  $C_{\bf w}^n$ be an $n$-cycle, where $n$ is odd. Then
\[
\depth (S/I(C_{\bf w}^n)^t)= 0 \text{\ for all\ } t\geq \frac{n+1}{2}.
\]
\end{Theorem}

\begin{proof}
	Let $I=I(C_{\bf w}^n)$ and $\mathfrak{m}=(x_1,x_2,\ldots,x_n)$. It suffices to prove $\mathfrak{m}\in Ass_S(I^t)$ by Lemma \ref{up}.
Choose $f=(x_1x_2)^{(t-\frac{n-1}{2}){\bf w}_1}\prod\limits^{n}_{j=3}x_j$, then $f\notin I^t$ and
 $(I^t:f)\subseteq \mathfrak{m}$. On the other hand, $x_1f=u_1^{t-\frac{n-1}{2}}\prod\limits^{\frac{n-1}{2}}_{k=1}u_{2k+1}\in I^t$ and $x_2f=u_1^{t-\frac{n-1}{2}}\prod\limits^{\frac{n-1}{2}}_{k=1}u_{2k}\in I^t$.
If $i$ is an integer with  $3\leq i\leq n$. We distinguish between two cases:
\begin{itemize}
\item[(a)] If $i$ is even, then
	$x_if=u_1^{t-\frac{n-1}{2}}(\prod\limits^{\frac{i}{2}-1}_{k=1}u_{2k+1})(\prod\limits^{\frac{n-1}{2}}_{m=\frac{i}{2}}u_{2m})\in I^t$.
\item[(b)] If $i$ is odd, then
$x_if=(x_1x_2)^{{\bf w}_1-1}u_1^{t-\frac{n-1}{2}-1}(\prod\limits^{\frac{i-1}{2}}_{k=1}u_{2k})(\prod\limits^{\frac{n-1}{2}}_{m=\frac{i-1}{2}}u_{2m+1})\in I^t$.
\end{itemize}
Therefore, $(I^t:f)\supseteq \mathfrak{m}$, which implies that $(I^t:f)=\mathfrak{m}$ and so $\mathfrak{m}\in Ass_S(I^t)$. The desired result follows.
\end{proof}

\begin{Theorem}\label{even1}
	Let  $C_{\bf w}^n$ be an $n$-cycle, where $n$ is even,  and ${\bf w}_1\geq 2$ and ${\bf w}_i=1$ for any $i\neq 1$. Let $t\geq\frac{n}{2}+1$ be an integer and  $f=(x_1x_2)^{(t-\frac{n}{2}){\bf w}_1}\prod\limits_{k=3}^{n}x_k$, then
	\[
	(I(C_{\bf w}^n)^t:f)=(\bigcup\limits_{k=1}^{\frac{n}{2}-1}\{x_1x_{2k}\})+(\bigcup\limits_{i=2}^{\frac{n}{2}-1}\bigcup\limits_{j=1}^{\frac{n}{2}-1}\{x_{2i+1}x_{2j}\})+(x_3,x_n).
	\]
\end{Theorem}
\begin{proof}
	Let $I=I(C_{\bf w}^n)$. First note that $x_3f=x_1^{{\bf w}_1-1}x_2^{{\bf w}_1-2}u_1^{t-\frac{n}{2}-1}u_2^2\prod\limits_{\ell=2}^{\frac{n}{2}}u_{2\ell}\in I^t$ and
$x_nf=x_1^{{\bf w}_1-2}x_2^{{\bf w}_1-1}u_1^{t-\frac{n}{2}-1}u_n^2\prod\limits_{\ell=1}^{\frac{n}{2}-1}u_{2\ell}\in I^t$.
Furthermore, if $k\in [\frac{n}{2}-1]$, then $x_1x_{2k}f=u_1^{t-\frac{n}{2}}\prod\limits_{\ell=1}^{k-1}u_{2\ell+1}\prod\limits_{\ell=k}^{\frac{n}{2}}u_{2\ell}\in I^t$, and if
$i\in [\frac{n}{2}-1]\backslash\{1\}$ and $j\in [\frac{n}{2}-1]$, then
	\[
	x_{2i+1}x_{2j}f=
	\begin{cases}
		u_1^{t-\frac{n}{2}}\prod\limits_{\ell=1}^{j-1}u_{2\ell+1}\prod\limits_{\ell=j}^{i}u_{2\ell}\prod\limits_{\ell=i}^{\frac{n}{2}-1}u_{2\ell+1}, &\text{if $i\geq j$,} \\ 	
		(x_1x_2)^{{\bf w}_1-1}u_1^{t-\frac{n}{2}-1}\prod\limits_{\ell=1}^{i}u_{2\ell}\prod\limits_{\ell=i}^{j-1}u_{2\ell+1}\prod\limits_{\ell=j}^{\frac{n}{2}}u_{2\ell}, &\text{if $i<j$.}
	\end{cases}
	\]
This implies that $x_{2i+1}x_{2j}f\in I^t$. Therefore, the inclusion relation ``$\supseteq$" holds.

On the other hand, let $u \in \mathcal{G}(I^n: f)$ be a monomial,
 then  $uf=h\prod\limits_{i=1}^{n}u_i^{k_i}$ for some  monomial $h$, where $\sum\limits_{i=1}^{n}k_i=t$ and   each $k_i\geq 0$.
	Thus
	\begin{align}\label{eqn: equality20}
		u(x_1x_2)^{(t-\frac{n}{2}){\bf w}_1}\prod\limits_{k=3}^{n}x_k=h\prod\limits_{i=1}^{n}u_i^{k_i}.
	\end{align}
	
	Now we prove that the inclusion relation ``$\subseteq$" holds by induction on $t$.

 If $t=\frac{n}{2}+1$,  then Eq. (\ref{eqn: equality20}) becomes
	\begin{align}\label{eqn: equality21}
		u(x_1x_2)^{{\bf w}_1}\prod\limits_{k=3}^{n}x_k=h\prod\limits_{i=1}^{n}u_i^{k_i}.
	\end{align}
If $k_2\geq 2$ or  $k_n\geq 2$,  then $u\in(x_3,x_n)$. If $k_i\geq 2$ for some $i\in [n-1]\backslash\{2\}$, then  $u\in (u_i)\subseteq (\bigcup\limits_{k=1}^{\frac{n}{2}-1}\{x_1x_{2k}\})+(\bigcup\limits_{i=2}^{\frac{n}{2}-1}\bigcup\limits_{j=1}^{\frac{n}{2}-1}\{x_{2i+1}x_{2j}\})$.
In the following, we consider the case where $0\leq k_i\leq 1$ for any $i\in [n]$ and divide into two cases:

 (I) If $k_1=1$, then Eq. (\ref{eqn: equality21}) becomes $u\prod\limits_{k=3}^{n}x_k=h\prod\limits_{i=2}^{n}u_i^{k_i}$ with $\sum\limits_{i=2}^{n}k_i=t-1$. There are four subcases:

(a) If $k_2=k_n=1$, then $u\in(x_1x_2)$.

 (b) If $k_2=1$ and $k_n=0$, then Eq. (\ref{eqn: equality21}) becomes  $u\prod\limits_{k=4}^{n}x_k=hx_2\prod\limits_{i=3}^{n-1}u_i^{k_i}$ with  $\sum\limits_{i=3}^{n-1}k_i=t-2$.
 Hence $k_{2i}=k_{2i+1}=1$ for some $i\in [\frac{n}{2}-1]$. Indeed, if $k_{2i}=0$ or $k_{2i+1}=0$ for any $i\in [\frac{n}{2}-1]$, then  $\sum\limits_{i=3}^{n-1}k_i\le (n-3)-(\frac{n}{2}-1)=\frac{n}{2}-2=t-3$, a contradiction. By comparing the degrees of $x_{2}$ and $x_{2i+1}$ in the above expression, we  can get $u\in(\bigcup\limits_{i=1}^{\frac{n}{2}-1}\{x_2x_{2i+1}\})\subseteq(\bigcup\limits_{i=2}^{\frac{n}{2}-1}\bigcup\limits_{j=1}^{\frac{n}{2}-1}\{x_{2i+1}x_{2j}\})+(x_3)$.

(c) If $k_2=0$ and $k_n=1$, then similar to the discussion of the case (b), we deduce that $u\in(\bigcup\limits_{i=1}^{\frac{n}{2}}\{x_1x_{2i}\})\subseteq(\bigcup\limits_{i=1}^{\frac{n}{2}-1}\{x_1x_{2i}\})+(x_n)$.

(d) If $k_2=k_n=0$,  then  Eq. (\ref{eqn: equality21}) becomes
 	\begin{align}\label{eqn: equality22}
 		u\prod\limits_{k=3}^{n}x_k=h\prod\limits_{i=3}^{n-1}u_i^{k_i},
 	\end{align}
 where $\sum\limits_{i=3}^{n-1}k_i=t-1$. Likewise,   we can get that $k_{2i-1}=k_{2i}=1$ for some $2\leq i\leq \frac{n}{2}-1$. Let $a_1=\min\{2\leq i\leq \frac{n}{2}-1\mid k_{2i-1}=k_{2i}=1\}$. We consider two subcases:

(i) If $a_1=2$, then  $k_{2b_1}=k_{2b_1+1}=1$ for  some $2\leq b_1\leq\frac{n}{2}-1$ by arguments similar to case (b) above.  By comparing the degrees of $x_{2a_1}$ and $x_{2b_1+1}$ in Eq. (\ref{eqn: equality22}), we can get that $u\in (x_{2a_1}x_{2b_1+1})\subseteq(\bigcup\limits_{i=2}^{\frac{n}{2}-1}\bigcup\limits_{j=1}^{\frac{n}{2}-1}\{x_{2i+1}x_{2j}\})$.

(ii) If $3\leq a_1\leq \frac{n}{2}-1$, then by the definition of $a_1$, we have $k_{2i-1}=0$ or $k_{2i}=0$ for any $2\leq i\leq a_1-1$, and $k_{2b_2}=k_{2b_2+1}=1$ for some $b_2\in [\frac{n}{2}-1]\backslash[a_1-1]$. Indeed, if
$k_{2b_2}=0$ or $k_{2b_2+1}=0$, then $\sum\limits_{i=3}^{n-1}k_i\leq (n-3)-(a_1-2)-(\frac{n}{2}-a_1)=\frac{n}{2}-1=t-2$, a contradiction.
By comparing the degrees of $x_{2a_1}$ and $x_{2b_2+1}$ in Eq. (\ref{eqn: equality22}), we can get that $u\in(x_{2a_1}x_{2b_2+1})\subseteq(\bigcup\limits_{i=2}^{\frac{n}{2}-1}\bigcup\limits_{j=1}^{\frac{n}{2}-1}\{x_{2i+1}x_{2j}\})$.

(II) If $k_1=0$, then Eq. (\ref{eqn: equality21}) becomes
\begin{align}\label{eqn: equality23}
	u(x_1x_2)^{{\bf w}_1}\prod\limits_{k=3}^{n}x_k=h\prod\limits_{i=2}^{n}u_i^{k_i},
\end{align}
where $\sum\limits_{i=2}^{n}k_i=t$. Likewise,  we have $k_{2i-1}=k_{2i}=1$ for some $2\leq i\leq \frac{n}{2}$. Let $a_2=\min\{2\leq i\leq \frac{n}{2}\mid k_{2i-1}=k_{2i}=1\}$. we divide into two subcases:

(i) If $a_2=2$, then  $k_{2b_3}=k_{2b_3+1}=1$ for  some $b_3\in [\frac{n}{2}-1]$ by arguments similar
to case (b) above. By comparing the degrees of $x_{2a_2}$ and $x_{2b_3+1}$ in Eq. (\ref{eqn: equality23}), we can deduce that $u\in (x_{2a_2}x_{2b_3+1})\subseteq(\bigcup\limits_{i=2}^{\frac{n}{2}-1}\bigcup\limits_{j=1}^{\frac{n}{2}-1}\{x_{2i+1}x_{2j}\})+(x_3)$.

(ii) If $3\leq a_2\leq \frac{n}{2}$, then  by the definition of $a_2$, we have  $k_{2i-1}=0$ or $k_{2i}=0$ for any $2\leq i\leq a_2-1$, and $k_{2b_4}=k_{2b_4+1}=1$ for some $b_4\in [\frac{n}{2}-1]\backslash[a_2-2]$.
Indeed, if $k_{2b_4}=0$ or $k_{2b_4+1}=0$, then
$\sum\limits_{i=2}^{n}k_i\leq (n-1)-(a_2-2)-(\frac{n}{2}-a_2+1)=\frac{n}{2}=t-1$, a contradiction.
 By comparing the degrees of $x_{2a_2}$ and $x_{2b_4+1}$ in Eq. (\ref{eqn: equality23}), we can get that $u\in(x_{2a_2}x_{2b_4+1})\subseteq(\bigcup\limits_{i=2}^{\frac{n}{2}-1}\bigcup\limits_{j=1}^{\frac{n}{2}-1}\{x_{2i+1}x_{2j}\})+(x_n)$.

Therefore, if $t=\frac{n}{2}+1$, then the desired conclusion holds.

In the following, we assume that  $t\geq \frac{n}{2}+2$. We consider two cases:

(i) If $k_1\geq 1$, then
$(I^t:f)\subseteq (I^{t-1}:\frac{f}{(x_1x_2)^{{\bf w}_1}})\subseteq (\bigcup\limits_{k=1}^{\frac{n}{2}-1}\{x_1x_{2k}\})+ (\bigcup\limits_{i=2}^{\frac{n}{2}-1}\bigcup\limits_{j=1}^{\frac{n}{2}-1}\{x_{2i+1}x_{2j}\})+(x_3,x_n)$ by Eq.  (\ref{eqn: equality20}) and  induction.

(ii) If $k_1=0$,
	then Eq. (\ref{eqn: equality20}) becomes
	\begin{align}\label{eqn: equality24}
		u(x_1x_2)^{(t-\frac{n}{2}){\bf w}_1}\prod\limits_{k=3}^{n}x_k=h\prod\limits_{i=2}^{n}u_i^{k_i},
	\end{align}
	where $\sum\limits_{i=2}^{n}k_i=t$. Likewise,  we have $k_{2i-1}=k_{2i}=1$ for some $2\leq i\leq \frac{n}{2}$. Let $a_3=\min\{2\leq i\leq \frac{n}{2}\mid k_{2i-1}=k_{2i}=1\}$. We divide into two subcases:

(i) If $a_3=2$, then  $k_{2b_5}=k_{2b_5+1}=1$ for some $b_5\in [\frac{n}{2}-1]$  by arguments similar
to the case $a_2=2$ above. By comparing the degrees of $x_{2a_3}$ and $x_{2b_5+1}$ in Eq. (\ref{eqn: equality24}), we can deduce that $u\in (x_{2a_3}x_{2b_5+1})\subseteq(\bigcup\limits_{i=2}^{\frac{n}{2}-1}\bigcup\limits_{j=1}^{\frac{n}{2}-1}\{x_{2i+1}x_{2j}\})+(x_3)$.

(ii) If $3\leq a_3\leq \frac{n}{2}$, then,  by arguments similar
to the case $3\leq a_2\leq \frac{n}{2}$ above, we deduce that $k_{2i-1}=0$ or $k_{2i}=0$ for any $2\leq i\leq a_3-1$, and $k_{2b_6}=k_{2b_6+1}=1$ for some $b_6\in [\frac{n}{2}-1]\backslash[a_3-2]$.
By comparing the degrees of $x_{2a_3}$ and $x_{2b_6+1}$ in Eq. (\ref{eqn: equality24}), we can get that $u\in(x_{2a_3}x_{2b_6+1})\subseteq(\bigcup\limits_{i=2}^{\frac{n}{2}-1}\bigcup\limits_{j=1}^{\frac{n}{2}-1}\{x_{2i+1}x_{2j}\})+(x_n)$.
\end{proof}

 For a monomial $u=x_1^{a_1}\cdots x_n^{a_n}$, we set $\deg_{x_i}(u)=a_i$ and $\supp(u)=\{x_i\mid a_i>0\}$. For a monomial ideal $I$, we set $\supp(I)=\bigcup\limits_{u\in \mathcal{G}(I)}\supp(u)$.

     \begin{Theorem}\label{even2}
		Let $t$ and $C_{\bf w}^n$ be  as in Theorem \ref{even1}. Then
$\depth(S/I(C_{\bf w}^n)^t)=1$.
 \end{Theorem}
	\begin{proof}
   Let $I=I(C_{\bf w}^n)$ and $f=(x_1x_2)^{(t-\frac{n}{2}){\bf w}_1}\prod\limits_{k=3}^{n}x_k$. Then  by Theorem \ref{even1}, we have
    \begin{align*}
(I^t:f)&=(\bigcup\limits_{k=1}^{\frac{n}{2}-1}\{x_1x_{2k}\})+(\bigcup\limits_{i=2}^{\frac{n}{2}-1}\bigcup\limits_{j=1}^{\frac{n}{2}-1}\{x_{2i+1}x_{2j}\})+(x_3,x_n)\\
&=(x_1,\hat{x}_3,x_5,\ldots,x_{n-1})(x_2,x_4,\ldots,x_{n-2},\hat{x}_n)+(x_3,x_n).
\end{align*}
where $\widehat{x_3}$ denotes the element $x_3$ omitted from  $\{x_1,x_3,x_5,\ldots,x_{n-1}\}$. Thus
	\[
	\depth(S/I^t)\leq \depth(S/(I^t:f))=1
	\]
by Lemmas \ref{sum1} and \ref{up}.
	
Next, we will prove that $\depth(S/I^t)\geq 1$.
Let $L^{(t)}=\{L_1^{(t)}, L_2^{(t)}, \ldots, L_r^{(t)}\}$ be a specific totally ordered set of all elements in $\mathcal{G}(I^t)$ as defined  in \cite[Setting 4.5]{ZCLY}. For simplicity, we set $L_i=L_i^{(1)}$. Let
  $C=\{L_i^{(t)}: L_1|^{edge}L_i^{(t)}\ \text{for\ }i\in[r]\}$ be a set of $L_i^{(t)}$   divided by  $L_1$ as an edge and $c=|C|$.  For each $i\in [c]$, we write $L_i^{(t)}$ as $L_{i}^{(t)}= L_{i_1}^{a_{i_1}}L_{i_2}^{a_{i_2}}\cdots L_{i_{k_i}}^{a_{i_{k_i}}}$ with $i_1=1$ and $2 \leq  i_2< \cdots < i_{k_{i}} \leq n$, where $\sum\limits_{j=1}^{k_i}a_{i_{j}}=t$  with  $a_{i_{j}}>0$ for all $j \in[k_i]$.
Assume that $J_i=(L_{i+1}^{(t)}, \ldots, L_r^{(t)})$ for all $i\in [c]$.
From the proof of \cite[Theorem 4.11]{ZCLY}, we obtain  $J_c=I(P_{\bf w}^n)^t$, where $I(P_{{\bf w}}^n)=(x_2x_3,x_3x_4,\ldots,x_{n-1}x_n,x_nx_1)$. So
	\begin{align*}
		\depth(S/J_c)=\depth(S/I(P_{\bf w}^n)^t)=\max\{\lceil\frac{n-t+1}{3}\rceil,1\}\geq 1
	\end{align*}
   by Lemma \ref{path}.  Let $s_i$ be the degree of $L_{i}^{(t)}$ for each $i\in [c]$.
 By the exact sequences
\begin{gather*}
	\begin{matrix}
		0 & \longrightarrow & \frac{S}{(J_1:L_{1}^{(t)})}(-s_1)  & \stackrel{ \cdot L_{1}^{(t)}} \longrightarrow  & \frac{S}{J_1} & \longrightarrow & \frac{S}{I^t} & \longrightarrow & 0 \\
		0 & \longrightarrow & \frac{S}{(J_2:L_{2}^{(t)})}(-s_2)  & \stackrel{ \cdot L_{2}^{(t)}} \longrightarrow  & \frac{S}{J_2} & \longrightarrow & \frac{S}{J_1} & \longrightarrow & 0  \\
		0 & \longrightarrow & \frac{S}{(J_3:L_{3}^{(t)})}(-s_3) & \stackrel{ \cdot L_{3}^{(t)}} \longrightarrow & \frac{S}{J_3} &\longrightarrow & \frac{S}{J_2} & \longrightarrow & 0 \\
		&  &\vdots&  &\vdots&  &\vdots&  \\
		0&  \longrightarrow & \frac{S}{(J_{c}:L_{c}^{(t)})}(-s_c) & \stackrel{ \cdot L_{c-1}^{(t)}} \longrightarrow & \frac{S}{J_{c}}& \longrightarrow & \frac{S}{J_{c-1}}& \longrightarrow & 0,
	\end{matrix}
\end{gather*}
we know that $\depth(S/I^t)\geq 1$ if $\depth(S/(J_i:L_i^{(t)}))\geq 2$ for any $i\in[c]$. It remains to prove that $\depth(S/(J_i:L_i^{(t)}))\geq 2$ for any $i\in[c]$.

 By \cite[Theorems 4.10]{ZCLY},  we obtain that   for each $i\in [c]$, $(J_i:L_{i}^{(t)})=K_i+Q_i$, where
	 $K_i=(L_4,L_5,\ldots,L_{n-2})+\sum\limits_{j=2}^{p_i}(x_{{i_j}+2})+(x_3,x_n)$,  $Q_{i}=\sum\limits_{j=0}^{q_i}(x_{n-2j})$, $p_i=\begin{cases}
		k_i-1,& \text{if $i_{k_i}=n$,}\\
		k_i,& \text{otherwise,}\\
	\end{cases}$ and $q_i=\max\{\ell_i:L_{n+1-2s}|^{edge}L_{i}^{(t)}$ for all $0\leq s\leq \ell_i\}$.

If $k_i=1$, or $k_i=2$ and $i_{k_i}=n$, then  $(J_1:L_{1}^{(t)})=(L_4,L_5,\ldots,L_{n-2})+(x_3,x_n)$. Thus $\depth(S/(J_1:L_1^{(t)}))\geq 2$ by Lemma \ref{path}. In the following, we consider the case $k_i\geq 3$,  or $k_i=2$ and $i_{k_i}<n$.

Note that if $i_{k_i}=n$, then $p_i={k_i}-1$, thus $2\le i_{p_i}<i_{k_i}=n$ by the choice of $k_i$, and if $i_{k_i}<n$, then  $2\le i_{p_i}=i_{k_i}<n$.
So $2\leq i_j\leq n-1$ for any $2\leq j\leq p_i$, which forces $i_j+2\in \{4,5,\ldots,n\}\cup\{n+1\}$. It follows that  $x_2\notin\mathcal{G}(K_i)$.
Next we prove that  $x_2\notin\mathcal{G}(Q_i)$, which forces $x_2\notin\mathcal{G}(J_i:L_i^{(t)})$.

Conversely, if $x_2\in\mathcal{G}(Q_i)$, then  $q_i=\frac{n}{2}-1$ by the expression of $Q_i$. By the definition of $q_i$,
 $L_{n+1-2s}\mid^{edge}L_i^{(t)}$ for all $0\leq s \leq \frac{n}{2}-1$. So we can write $L_i^{(t)}$ as
$L_i^{(t)}=h_1\prod\limits^{\frac{n}{2}}_{s=1}L_{n+1-2s}$ for some monomial $h_1\in \mathcal{G}(I^{t-\frac{n}{2}})$. This implies $h_1\prod\limits^{\frac{n}{2}}_{k=1}L_{2k}\in \mathcal{G}(I^t)$.
By the expression of $L_i^{(t)}$, we know that $L_i^{(t)}=(x_1x_2)^{{\bf w}_1-1}v$  with $v=h_1\prod\limits^{\frac{n}{2}}_{k=1}L_{2k}$, which contradicts with $L_i^{(t)}\in \mathcal{G}(I^t)$.

By the definition of $Q_i$ and $K_i$,  we have $(x_1x_2)^{{\bf w}_1}, x_2x_3\notin\mathcal{G}(Q_i)$,  $(x_1x_2)^{{\bf w}_1}, x_2x_3\notin\mathcal{G}(K_i)$, which means that $ x_2x_3,(x_1x_2)^{{\bf w}_1}\notin\mathcal{G}(J_i:L_i^{(t)})$.
Thus $x_2\notin\supp(J_i:L_i^{(t)})$ for any $i\in [c]$. Therefore,
 $\depth(S/(J_i:L_i^{(t)}))\geq 1$.
Note that $(J_i:L_i^{(t)})=(L_4,L_5,\ldots,L_{n-2})+\sum\limits_{j=2}^{p_i}(x_{{i_j}+2})+\sum\limits_{j=0}^{q_i}(x_{n-2j})+(x_3,x_n)$, it follows from Lemma \ref{path} that
 $\depth(S/(J_i:L_i^{(t)}))=1$ if and only if $(J_i:L_i^{(t)})=(x_1,\hat{x}_2,x_3,\ldots,x_n)$  for any $i\in[c]$. Therefore, it is remains to  prove $(J_i:L_i^{(t)})\ne (x_1,\hat{x}_2,x_3,\ldots,x_n)$. In fact, if   $(J_i:L_i^{(t)})=(x_1,\hat{x}_2,x_3,\ldots,x_n)$ for some $i\in[c]$, then  $x_{2k-1}\in (J_i:L_i^{(t)})$ for any $k\in [\frac{n}{2}]$.
 Since   for any $k\in\{1\}\cup\{3,4,\ldots,\frac{n}{2}\}$, $x_{2k-1}\notin(L_4,L_5,\ldots,L_{n-2})+\sum\limits_{j=0}^{q_i}(x_{n-2j})+(x_3,x_n)$, we have  $x_{2k-1}\in  \sum\limits_{j=2}^{p_i}(x_{{i_j}+2})$.  Depending on the range of $k$,  we see that $\bigcup\limits_{k=1}^{\frac{n}{2}-1}\{2k+1\}\subseteq \bigcup\limits_{j=2}^{p_i}\{i_j\}$. Since $ p_i\le k_i$,  $L_{i_j}\mid^{edge}L_i^{(t)}$ for all $2\le j\le p_i$ by the expression $L_{i}^{(t)}=L_{i_1}^{a_{i_1}}L_{i_2}^{a_{i_2}}\cdots L_{i_{k_i}}^{a_{i_{k_i}}}$.  Thus $L_{2k+1}\mid^{edge}L_i^{(t)}$ for all $k\in[\frac{n}{2}-1]$. Again by the definition of $C$, we get $L_1\mid^{edge}L_i^{(t)}$.
So  $L_i^{(t)}$ can be expressed as
$L_i^{(t)}=h_2\prod\limits^{\frac{n}{2}}_{s=1}L_{2s+1}$ for some monomial $h_2\in \mathcal{G}(I^{t-\frac{n}{2}})$.
Note that $L_i^{(t)}$ can also be written as $L_i^{(t)}=(x_1x_2)^{{\bf w}_1-1}v$  with $v=h_2\prod\limits^{\frac{n}{2}}_{k=1}L_{2k}$, which contradicts with $L_i^{(t)}\in \mathcal{G}(I^t)$.

Therefore,   $\depth(S/(J_i:L_i^{(t)}))\geq 2$ and we finish the proof.
 \end{proof}

\medskip
\hspace{-6mm} {\bf Acknowledgments}

 \vspace{3mm}
\hspace{-6mm}  This research is supported by the Natural Science Foundation of Jiangsu Province
(No. BK20221353) and the National NaturalScience Foundation of China (No.
12471246). The authors are grateful to the software systems \cite{Co} and \cite{GS}
 for providing us with a large number of examples.

	\end{document}